\documentclass[a4paper, 12pt]{amsart}
\usepackage{graphicx}
\usepackage{chngcntr}
\usepackage{amsmath}
\usepackage{amssymb}
\usepackage{amsfonts}
\usepackage{verbatim}
\usepackage{scalefnt}
\usepackage{multirow}
\usepackage{mathtools}
\usepackage{float}
\usepackage{multicol}
\usepackage{blkarray}
\restylefloat{figure}
\usepackage[margin=1.0in]{geometry}

\usepackage[usenames,dvipsnames]{xcolor}
\usepackage{pstricks,tikz}
\usepackage{pstricks,tikz}

\usepackage{fancyhdr}
\usepackage{hyperref}
\usepackage{amsaddr}
\usepackage{bm}

\newcommand{\ep}{\varepsilon}
\newcommand{\de}{\delta}
\newcommand{\floor}[1]{\left\lfloor {#1}\right\rfloor}
\newcommand{\ceil}[1]{\left\lceil {#1}\right\rceil}

\newenvironment{claimproof}[1]{\par\noindent\underline{Proof:}\space#1}{\leavevmode\unskip\penalty9999 \hbox{}\nobreak\hfill\quad\hbox{$\blacksquare$}}

\DeclareMathOperator*{\Odd}{{\rm O}}

\begin{document}

\def\COMMENT#1{}
\let\COMMENT=\footnote% COMMENT OUT for clean output

\newtheorem{case}{Case}
\newtheorem{subcase}{Case}[case]
\counterwithin*{case}{subsection} 
\newtheorem{problem}{Problem}
\newtheorem{theorem}{Theorem}
\newtheorem{lemma}[theorem]{Lemma}
\newtheorem{proposition}[theorem]{Proposition}
\newtheorem{corollary}[theorem]{Corollary}
\newtheorem{conjecture}[theorem]{Conjecture}
\newtheorem{claim}[theorem]{Claim}
\newtheorem{definition}[theorem]{Definition}
\newtheorem{observation}[theorem]{Observation}
\newtheorem{question}[theorem]{Question}
\newtheorem{remark}[theorem]{Remark}

\numberwithin{equation}{section}
\numberwithin{theorem}{section}

\def\eps{{\varepsilon}}
\newcommand{\cC}{\mathcal{C}}
\newcommand{\cF}{\mathcal{F}}
\newcommand{\cP}{\mathcal{P}}
\newcommand{\cT}{\mathcal{T}}
\newcommand{\cL}{\mathcal{L}}
\newcommand{\ex}{\mathbb{E}}
\newcommand{\eul}{e}
\newcommand{\pr}{\mathbb{P}}
\newcommand{\feas}{\mathop{\textsc{feas}}}

\title[Integer colourings with forbidden monochromatic sums]{On the maximum number of integer colourings with forbidden monochromatic sums}
\author{Hong Liu, Maryam Sharifzadeh and Katherine Staden}
\address{Mathematics Institute and DIMAP, University of Warwick, Coventry CV4 7AL, UK}
\email{\tt{$\lbrace$h.liu.9,~m.sharifzadeh,~k.l.staden$\rbrace$@warwick.ac.uk}}
\thanks{H.L.\ was supported by EPSRC grant~EP/K012045/1 and ERC
	grant~306493 and the Leverhulme Trust
	Early Career Fellowship~ECF-2016-523.
M.S.\ and K.S.\ were supported by ERC
	grant~306493.}

\begin{abstract}
	Let $f(n,r)$ denote the maximum number of colourings of $A \subseteq \lbrace 1,\ldots,n\rbrace$ with $r$ colours such that each colour class is sum-free.
	Here, a sum is a subset $\lbrace x,y,z\rbrace$ such that $x+y=z$.
	We show that $f(n,2) = 2^{\lceil n/2\rceil}$, and describe the extremal subsets.
	Further, using linear optimisation, we asymptotically determine the logarithm of $f(n,r)$ for $r \leq 5$.
	Similar results were obtained by H\`an and Jim\'enez in the setting of finite abelian groups.
   %The problem of determining $f(n,r)$ is a translation of a problem of Erd\H{o}s-Rothschild concerning colourings of graphs with forbidden monochromatic cliques into the sum-free setting.
  % Using a container lemma of Green, we translate the problem into a certain optimisation problem, which is solved with the aid of a stability theorem of ????? and duality.
\end{abstract}

\date{\today}
\maketitle

\section{Introduction and results}

A recent trend in combinatorial number theory has been to consider versions of classical problems from extremal graph theory in the sum-free setting.
We state some examples.
The famous theorem of Mantel from 1907~\cite{mantel} states that every $n$-vertex graph with more than $\lfloor n^2/4\rfloor$ edges 
%(so, roughly speaking, with edge density more than $1/2$) 
necessarily contains a triangle.
On the other hand, it is not hard to show that every subset $A$ of $[n] := \lbrace 1,\ldots,n\rbrace$ of size more than $\lceil n/2\rceil$ necessarily contains a \emph{Schur triple}, or \emph{sum}; that is, a triple $\lbrace x,y,z\rbrace$ of not necessarily distinct elements such that $x+y=z$.
The name dates back to a result of Schur from 1916 which states that every $r$-colouring of $[n]$ contains yields a monochromatic Schur triple when $n$ is sufficiently large.
Its graph-theoretic counterpart is Ramsey's theorem from 1928 which guarantees a monochromatic clique in any $r$-edge-colouring of a sufficiently large complete graph.
The triangle removal lemma of Ruzsa and Szemer\'edi~\cite{triangleremoval} states that every $n$-vertex graph containing $o(n^3)$ triangles can be made triangle-free by removing $o(n^2)$ edges.
In the sum-free setting, the removal lemma of Green~\cite{G-R}, and Kr\'al', Serra and Vena~\cite{ksv} states that every $A \subseteq [n]$ containing $o(n^2)$ Schur triples can be made sum-free by removing $o(n)$ elements.
Erd\H{o}s, Kleitman and Rothschild~\cite{ekr} proved that the number of $n$-vertex triangle-free graphs is $2^{n^2/4+o(n^2)}$, that is, the obvious lower bound of taking every subgraph of a maximal triangle-free graph is, in a sense, tight.
In the sum-free setting, resolving a conjecture of Cameron and Erd\H os~\cite{CE-SFConj}, Green~\cite{G-CE} and independently Sapozhenko~\cite{sap} proved that, for every $i=0,1$, there exists a constant $C_i$, depending only on the parity of $n$, such that $[n]$ contains  $(C_i+o(1))2^{n/2}$ sum-free sets. So again, the obvious lower bound is tight.

\begin{comment}
Indeed, there are many interesting questions in combinatorial number theory which can, roughly speaking, be formulated by considering subsets of $[n]$ and Schur triples in place of $n$-vertex graphs and triangles respectively.
Of course, there are many variations on this theme, and there is no need to restrict to triangles and Schur triples, and indeed one can ask analogous questions concerning other discrete objects, other small graphs and other equations with few variables, see e.g.~\cite{}.
Rather, the above results answer some basic and general combinatorial questions in various settings, for example finding the largest member of a class $\mathcal{C}$ of discrete objects that does not contain some fixed $F$ as a substructure; or counting the members of $\mathcal{C}$ which do not contain $F$.
\end{comment}

\subsection{The Erd\H{o}s-Rothschild problem for cliques in graphs}

Taking inspiration from the extremal graph theory literature, in this paper we consider another classical graph problem in the sum-free setting: the problem of Erd\H{o}s and Rothschild~\cite{ER,ER2}, which is stated as follows.
Given an $n$-vertex graph $G$ and positive integers $r,k$, say that a colouring of its edges with $r$ colours (an \emph{$r$-edge-colouring}) is \emph{valid} if there are no monochromatic copies of $K_k$.
Among all such graphs $G$, what is the maximum number $F(n,r,k)$ of valid colourings? 
Ramsey's theorem implies that any graph containing a sufficiently large clique has no valid colourings.
Clearly, any colouring of a $K_k$-free graph is valid. Tur\'an's theorem implies that the largest such graph is $T_{k-1}(n)$, the complete balanced $(k-1)$-partite graph. Thus we obtain the bound
\begin{equation}\label{eq-tl}
	F(n,r,k) \geq r^{t_{k-1}(n)},
\end{equation}
where $t_{k-1}(n)$ is the number of edges in $T_{k-1}(n)$.
Erd\H{o}s and Rothschild conjectured that this trivial lower bound is tight when $n$ is large and $(r,k)=(2,3)$ (i.e.~one uses two colours and forbids monochromatic triangles), and further, that $T_2(n)$ is the unique extremal graph.
This was confirmed for all $n \geq 6$ by Yuster~\cite{yuster}, who himself extended the conjecture to larger cliques in the case $r=2$.
This was in turn verified by Alon, Balogh, Keevash and Sudakov~\cite{abks}, showing that for $r\in \{2,3\}$, equality holds in~\eqref{eq-tl}, and $T_{k-1}(n)$ is the unique extremal graph.
They further showed that for all $r,k \in \mathbb{N}$ such that $k \geq 3$ and $r \geq 2$, the limit
$
\lim_{n \rightarrow \infty}\log_r F(n,r,k)/\binom{n}{2}
$
exists.
In other words, there is some $h(r,k)$ such that $F(n,r,k) = r^{h(r,k) \binom{n}{2} + o(n^2)}$.
However, for every other choice of $(r,k)$, there exists a constant $c=c(r,k)$ such that $h(r,k) > (k-2)/(k-1) + c$.
Note that the edge density of $T_{k-1}(n)$ tends to $(k-2)/(k-1)$ with $n$, so this says that the trivial lower bound in~\eqref{eq-tl} is \emph{not} correct for all other choices of $(r,k)$.
Observe that this implies that extremal graphs must therefore contain many copies of the forbidden $K_k$. 
The authors were, however, able to determine $h(3,4)$ and $h(4,4)$.
The exact results in these cases were obtained by Pikhurko and Yilma~\cite{PY}, who showed that the unique extremal graphs for $F(n,3,4)$ and $F(n,4,4)$ are $T_4(n)$ and $T_9(n)$ respectively.

The problem remains unresolved, even asymptotically in the logarithmic, in all other cases.
In particular, there is no (approximate) solution when the number of colours $r$ is at least five, for any $k \geq 3$. A generalisation of the Erd\H{o}s-Rothschild problem was considered by Pikhurko, Yilma and the third author of this paper~\cite{psy}, wherein one may forbid differently-sized cliques for different colours.
A special case of the main result is that there is a certain finite optimisation problem whose maximum is equal to $h(r,k)$. Roughly speaking, they showed that to determine $F(n,r,k)$ up to an error in the exponent, one should maximise over all possible `layerings' of `small' $K_k$-free graphs.

\subsection{The Erd\H{o}s-Rothschild problem in the sum-free setting: our results}

As stated above, the purpose of this paper is to initiate the study of the Erd\H{o}s-Rothschild problem in the case when our underlying discrete structure is not $K_n$ but $[n]$, and the forbidden substructure is not a triangle but a Schur triple.
Let us formulate the problem precisely.

\medskip
\noindent

\begin{problem}\label{prob-main}
\emph{Given positive integers $n,r$, determine $f(n,r)$}, defined as follows.
For each $A \subseteq [n]$, say that a colouring $\sigma : A \rightarrow [r]$ of $A$ with $r$ colours is \emph{valid} if it contains no monochromatic sums.
In other words, $\sigma^{-1}(c)$ is a sum-free set for all colours $c$.
Let $f(A,r)$ be the number of valid colourings of $A$ and let
$$
f(n,r) := \max_{A \subseteq [n]}f(A,r).
$$
\end{problem}
\medskip
\noindent
Notice that Schur's theorem says that $f([n],r)=0$ whenever $n$ is sufficiently large (as a function of $r$).
Before stating our results, let us see what we might conjecture via analogy with the graph setting.
Here, for $r \in \lbrace 2,3\rbrace$, the unique extremal graph was the largest triangle-free graph.
It is not hard to see that every largest sum-free sets in $[n]$ has size at most $\lceil n/2 \rceil$, and the only subsets attaining this bound are
\begin{itemize}
\item[(i)] $\Odd := \lbrace 1,3,5,\ldots,2\lceil n/2 \rceil-1\rbrace$ and
\item[(ii)] $I_2 := \lbrace \lfloor n/2 \rfloor+1,\lfloor n/2\rfloor+2,\ldots,n\rbrace$.
\end{itemize}
Additionally, if $n$ is even, the set $\{n/2,\ldots,n-1\}$ is also a largest sum-free subset.

Our first theorem is an exact result for two colours, which states that the largest sum-free subsets of $[n]$ are also extremal subsets for Problem~\ref{prob-main}.

\begin{theorem}\label{2colours}
There exists $n_0 >0$ such that for all integers $n \geq n_0$, we have
$$
f(n,2) = 2^{\lceil n/2\rceil}.
$$
Moreover, the only extremal subsets are $\Odd$, $I_2$; and if $n$ is even, we additionally have $\{n/2,\ldots,n-1\},\{n/2,\ldots,n\}$.
\end{theorem}

Note that, unlike in the graph case, in the sum-free setting with two colours, there are extremal subsets which are \emph{not} sum-free. However, they all contain at most one sum.

Unsurprisingly, and as in the graph setting, our results decrease in strength as the number of colours increases.
When $r=3$ we can obtain the following stability theorem.

\begin{theorem}\label{3colours}
For all positive integers $n$, we have
$$
f(n,3) = 3^{n/2+o(n)}.
$$
Moreover, the following holds.
For all $\eps>0$ there exist $\delta,n_0>0$ such that for all integers $n \geq n_0$, whenever $A \subseteq [n]$ satisfies $f(A,3) \geq f(n,3) \cdot 2^{-\delta n}$, we have that either $|A \bigtriangleup \Odd| \leq \eps n$; or $|A \bigtriangleup I_2| \leq \eps n$. 
\end{theorem}

Finally, we asymptotically determine the logarithm of $f(n,r)$ when $r \in \lbrace 4,5\rbrace$.

\begin{theorem}\label{45colours}
For $r \in \lbrace 4,5\rbrace$ and all positive integers $n$, we have
$$
f(n,r) = \left(r\left\lfloor \frac{r^2}{4}\right\rfloor\right)^{\frac{n}{4}+o(n)}.
$$
\end{theorem}

In particular, we are able to asymptotically solve the $5$-colour case of the Erd\H{o}s-Rothschild problem in the sum-free setting, in contrast to the graph setting in which it is wide open for triangles and larger cliques.
(This turns out to be a consequence of the rigid structure of large maximal sum-free sets.)
The (asymptotic) lower bounds in Theorems~\ref{2colours}--\ref{45colours} come from the facts that
$$
f(\Odd,r)=f(I_2,r) = r^{\lceil n/2\rceil}\quad\text{and}\quad f(\Odd \cup I_2,r) \geq \left(r\left\lfloor \frac{r^2}{4}\right\rfloor\right)^{\left\lfloor\frac{n}{4}\right\rfloor}.
$$
The first assertion follows from the fact that every $r$-colouring of a sum-free set $A$ is valid, so $f(A,r)=r^{|A|}$.
For the second, note that any colouring $\sigma : \Odd\cup I_2 \rightarrow [r]$, with $\sigma(x) \in \lbrace 1,\ldots,\lfloor r/2\rfloor\rbrace$ whenever $x$ is odd; and $\sigma(y) \in \lbrace \lfloor r/2\rfloor+1,\ldots,r\rbrace$ whenever $y > \lfloor n/2\rfloor$, is valid, giving the claimed bound.

Note that $r^{1/2} > (r\lfloor r^2/4\rfloor)^{1/4}$ for $r \in \lbrace 2,3\rbrace$, while this inequality becomes an equality for $r=4$, and reverses for $r \geq 5$.
It is tempting to believe that the bound in Theorem~\ref{45colours} holds for all $r \geq 4$.
However, this
is \emph{not} true. Indeed, when $r$ is large, the set $A := \Odd \cup I_2 \cup \lbrace x \in [n]: x \equiv 1,4 \mod 5\rbrace$ contains exponentially more valid $r$-colourings than $\Odd \cup I_2$. It would also be interesting to see if for every fixed integer $r \geq 6$, the limit $\lim_{n \rightarrow \infty} \log  f(n,r)/n$ exists.

\subsection{Some remarks on the methods and proofs}
An important tool in our proof is the Green's container theorem for finite abelian groups (Theorem~\ref{container}). 
The special case that we need states that, for every positive integer $n$, there is a small family $\mathcal{F}$ of subsets of $[n]$, called \emph{containers}, each of which is almost sum-free, and such that every sum-free subset of $[n]$ lies in some member of $\mathcal{F}$.

In the proof of the main result in~\cite{abks} and other Erd\H{o}s-Rothschild-type results for graphs, Szemer\'edi's regularity lemma~\cite{reg} is used to approximate a large graph $G$ by another graph of bounded size (the \emph{reduced graph}).
Then, for each valid $r$-edge-colouring $\sigma$ of $G$, for each $i \in [r]$ one can approximate the $K_k$-free subgraph $\sigma^{-1}(i)$ of $G$ by a $K_k$-free graph of bounded size.
In our proofs, for each valid $r$-colouring of $[n]$, we approximate the sum-free subset $\sigma^{-1}(i)$ of $[n]$ by a container.

We use Theorem~\ref{container} to reduce the problem of determining $f(n,r)$ to solving an optimisation problem (Problem~2) whose maximum approaches $\log  f(n,r)/n$ as $n$ tends to infinity (Theorem~\ref{problem2}).
Roughly speaking, Problem~2 involves layering sum-free subsets $A_1,\ldots,A_r$ of $[n]$ and measuring a weighted overlap $g(A_1,\ldots,A_r)$.

To attack Problem~2, we require a second important tool, namely a very strong stability theorem of Deshoulliers, Freiman, S\'os, and  Temkin~\cite{DFST}, which was recently strengthened by Tran~\cite{tuan}. This states that every sum-free subset of $[n]$ is either `small', or has a very rigid structure: either it contains only odd elements, or it somehow resembles the interval $I_2$.
Now it turns out that, if $r$ is small and $(A_1,\ldots,A_r)$ is a maximiser for Problem~2, then at most one of the $A_i$ can be small.
The rigid structure of the others means that the feasible set for Problem~2 is not too large.
In fact, for $r \in \lbrace 2,3\rbrace$, it can be easily solved at this stage, and we find that either all of $A_1,\ldots,A_r$ are close to $\Odd$; or they are all close to $I_2$.
This proves Theorem~\ref{3colours} and completes the first step of the proof of Theorem~\ref{2colours}.

To complete the proof of Theorem~\ref{2colours}, we have two cases to consider. The solution to Problem~2 when $r=2$ implies that any subset $A \subseteq [n]$ with $f(A,r)=f(n,r)$ satisfies either (1) $|A \bigtriangleup \Odd| = o(n)$; or (2) $|A \bigtriangleup I_2| = o(n)$. We use stability arguments, together with techniques from~\cite{BLST}, to obtain the exact structure of $A$.

For $r \in \lbrace 4,5\rbrace$, we find a reduction of Problem~2 to a \emph{linear} optimisation problem.
First, for each $i \in [r]$, and any feasible $(A_1,\ldots,A_r) \subseteq [n]^r$, we obtain $d_i \in [0,1]$ which are each functions of $A_1,\ldots,A_r$ and $n$ and such that $g(A_1,\ldots,A_r)$ is linear in $d_1,\ldots,d_r$.
Now, using the structural information returned from stability, we obtain constraints, linear in $d_1,\ldots,d_r$, which every maximiser $(A_1,\ldots,A_r)$ must satisfy.
This gives rise to a linear program in the variables $d_1,\ldots,d_r$.
Now, this linear program is a relaxation of Problem~2, so its maximiser may not correspond to a feasible solution $(A_1,\ldots,A_r)$ of Problem~2.
But, if we can exhibit a feasible tuple $(A_1,\ldots,A_r)$ such that the maximum $M$ of this program satisfies $M =g(A_1,\ldots,A_r)$, then $(A_1,\ldots,A_r)$ \emph{is} a maximiser of Problem~2.
Thus our task is to find enough constraints (of sufficient strength) so that this is possible.
In so doing, we will prove Theorem~\ref{45colours}.

%Annoying technical things: Difficult to pass to a finite optimisation problem. Can approximate a graph by a regularity partition, what to approximate $[n]$ by? And retain notions like odd/even.

\subsection{The Erd\H{o}s-Rothschild problem in other settings}\label{other}

Erd\H{o}s and Rothschild also considered the problem of counting monochromatic $H$-free colourings, for an arbitrary fixed graph $H$.
In~\cite{abks}, it is shown that the analogue of their main result for cliques in fact holds when $H$ is colour-critical.
Further cases including matchings, stars, paths, trees were investigated in~\cite{hkl2,hkl3}.
Other works have considered a fixed forbidden colour pattern of $H$, see~\cite{balogh,bhs,hl,hlos}.

An analogous problem for directed graphs was solved by Alon and Yuster~\cite{alonyuster}, who determined, for each $k$-vertex tournament $T$, the maximum number of $T$-free orientations of an $n$-vertex graph, when $n$ is sufficiently large.
The hypergraph analogue was addressed in~\cite{hkl,lprs,lps}.

The authors of~\cite{cdt} and~\cite{hkl} considered the problem of
counting the number of colourings of families of $r$-sets such that every colour class is $\ell$-intersecting.
A related result in the context of vector spaces over a finite field $GF(q)$ is proved in~\cite{qana}.
These results are Erd\H os-Rothschild  versions of the classical Erd\H{o}s-Ko-Rado theorem.

During the preparation of this paper, we became aware of the results of H\`an and Jim\'enez~\cite{hj} who recently studied similar questions in the setting of finite abelian groups.
Given a finite abelian group $(\Gamma,+)$, define an $r$-colouring of $A \subseteq \Gamma$ to be \emph{valid} if it has no monochromatic sum.
Let $f(\Gamma,r)$ be the maximum number of valid $r$-colourings among all subsets $A$ of $\Gamma$. The results of H\`an and Jim\'enez show a close relationship between $f(\Gamma,r)$ and the largest sum-free sets of $\Gamma$, and characterise for $r\le 5$ the extremal sets.
Their proof also uses the container lemma of Green in a similar way as described above. We remark that H\`an and Jim\'enez's result and ours do not imply one another. 
%As the stability theorem they use require $G$ to have a fixed size subgroup. Our result resembles more the cyclic group $\mathbb{Z}/p\mathbb{Z}$.

%Though the first part of the proof of Theorem~\ref{groups} uses similar methods to the first parts of the proofs of our results Theorems~\ref{2colours}--\ref{45colours}, neither sets of results can be derived from the other.

\subsection{Organisation of the paper}

Section~\ref{sec:notation} sets up the notation we will use and contains the statements of results on sum-free sets necessary for the proof.
In Section~\ref{sec:problem2} we define Problem~2, the optimisation problem whose maximum is a parameter $g(n,r)$ which is closely related to $f(n,r)$.
Then in Sections~\ref{sec:3},~\ref{sec:2} and~\ref{sec:45} we prove Theorems~\ref{3colours},~\ref{2colours} and~\ref{45colours} respectively.
We make some concluding remarks in Section~\ref{sec:conclude}.

\section{Notation and Preliminaries}\label{sec:notation}

In this section we define the notation that we will use, and some results on sum-free subsets which are needed in our proofs.

\subsection{Notation}

Given integers $m,n$ such that $m \leq n$, we write $[m,n]$ to denote the set $\lbrace m,\ldots,n\rbrace$, and write $[n] := [1,n]$.
For a set $A\subseteq [n]$, we define $d(A)=|A|/n$ and $\min(A)$ to be the \emph{density} and the minimum element of $A$, respectively. We also define $E$ and $\Odd$ to be respectively the set of all even and odd integers in $[n]$. 
As we defined earlier, $I_2 := [\lfloor n/2\rfloor+1,n]$ and $I_1 := [\lfloor n/2\rfloor]$.
(So we suppress the dependence on $n$ in the notation).
Given $A,B \subseteq \mathbb{Z}$, we write $A+B := \lbrace a+b : a \in A, b \in B\rbrace$.
For any $x \in \mathbb{Z}$, we also write $x \cdot A := \lbrace xa: a \in A\rbrace$. 
Logarithms will always to taken to the base $2$.

\subsection{Tools for sum-free subsets}

The first result we state is a very strong stability theorem for sum-free subsets due to Deshouillers, Freiman, Temkin and S\'os~\cite{DFST}.
It states that every large sum-free $S \subseteq [n]$ either contains no even number, or is, in a certain sense, close to the interval $I_2$.

\begin{theorem}[\cite{DFST}]\label{thm-sftypes}
	Every sum-free set $S$ in $[n]$ satisfies at least one of the following conditions:
	\begin{itemize}
		\item[(a)] $|S| \leq 2n/5+1$;
		\item[(b)] $S$ consists of odd numbers;
		\item[(c)] $|S| \leq \mathrm{min}(S)$.	
				\end{itemize}
\end{theorem}

%By Theorem~\ref{thm-sftypes}, every sum-free subset $A\in [n]$ either
% 
%\begin{itemize}
%	\item[(a)] $|A|< 2n/5$,
%	\item[(b)] $A$ consists of odd numbers, 
%	\item[(c)] $A\subseteq [2n/5, n]$ and $|A|\le\min(A)$.
%\end{itemize}

Throughout the rest of the paper, we refer to such sum-free sets as \emph{type~(a)}, \emph{type~(b)}, and \emph{type~(c)} respectively.

We use the following \emph{container theorem} of Green~\cite{G-R}, which, for large $n$, guarantees a small collection of subsets of $[n]$ which somehow  approximates the collection of sum-free sets. We should also mention that (hyper)graph containers have been used successfully in many contexts, see~\cite{container1,kw1,kw2,container2}.

\begin{theorem}[\cite{G-R}]\label{container}
	For all $\eps>0$ there exists $n_0>0$ such that, for all integers $n \geq n_0$, there exists a family $\mathcal{F}$ of subsets of $[n]$ with the following properties:
	\begin{itemize}
		\item[(i)] Every $F \in \mathcal{F}$ contains at most $\eps n^2$ Schur triples;
		\item[(ii)] If $S \subseteq [n]$ is sum-free, then $S \subseteq F$ for some $F \in \mathcal{F}$;
		\item[(iii)] $|\mathcal{F}| \leq 2^{\eps n}$;
		\item[(iv)] $|F| \leq (1/2+\eps)n$ for all $F \in \mathcal{F}$.
	\end{itemize}
\end{theorem}

Given a sum-free set $S \subseteq [n]$, the set $F \in \mathcal{F}$ guaranteed by~(ii) is called a \emph{container} for $S$.
We also need the following removal lemma of  Green~\cite{G-CE,G-R}, and Kr\'al', Serra and Vena~\cite{ksv}, which guarantees that a subset of $[n]$ containing $o(n^2)$ sums can be made sum-free by removing $o(n)$ elements.

\begin{theorem}[\cite{G-CE,G-R,ksv}]\label{removal}
For all $\eps>0$, there exists $\delta,n_0>0$ such that the following holds for all integers $n \geq n_0$.
	Suppose that $A \subseteq [n]$ is a set containing at most $\delta n^2$ Schur triples.
	Then there exist $B,C \subseteq [n]$ such that $A=B \cup C$ where $B$ is sum-free and $|C| \leq \eps n$.
	\end{theorem}

Finally, we will use the famous Cauchy-Davenport inequality which bounds the size of the set $A+B$:

\begin{theorem}\label{CD}
For all finite non-empty subsets $A,B$ of $\mathbb{Z}$, we have that $|A+B| \geq |A|+|B|-1$.
\end{theorem}

\section{An equivalent covering problem}\label{sec:problem2}

In this section, we define a new maximisation problem whose value $g(n,r)$ is closely related to $f(n,r)$. Then, for the rest of the paper, it suffices to consider this new problem.
To motivate the problem, consider the following procedure for finding a subset $A \subseteq [n]$ with many valid colourings.
Let $r \in \mathbb{N}$ be the number of colours, as usual, and choose sum-free subsets $A_1,\ldots,A_r$ of $[n]$.
Then the number of valid colourings of $\bigcup_{i \in [r]}A_i$ is at least the number of $\sigma$ which colours $x$ with some $i$ such that $x \in A_i$.
If $x$ lies in many $A_i$ then the number of choices for $\sigma(x)$ is large.
So a choice of $A_1,\ldots,A_r$ with a large appropriately weighted overlap generates many valid colourings.
We now make this precise.

\begin{problem}\label{prob-equivalent}
\emph{Given $n,r \in \mathbb{N}$, determine $g(n,r)$}, defined as follows.
Given a tuple $(A_1,\ldots,A_r)$ of sum-free subsets of $[n]$,
for each $I \in 2^{[r]}$, let $E_I := \bigcap_{i \in I}A_i\setminus \bigcup_{j \not\in I}A_j$ be the set of $x \in [n]$ which lie in $A_i$ if and only if $i \in I$.
Define
$$
g(A_1,\ldots,A_r) := \frac{1}{n}\sum_{I \in 2^{[r]}\setminus \lbrace \emptyset\rbrace}\left|E_I\right|\log  |I|.
$$
Equivalently, for each $i \in [r]$ let $D_i := \bigcup_{I \in 2^{[r]}:|I|=i} E_I$; that is, the set of all elements that are in exactly $i$ different $A_j$'s. %Therefore, $D_i$'s are pairwise disjoint.
Let $d_i := |D_i|/n$ and define $(d_1,\ldots,d_r)$ to be the \emph{intersection vector} of $(A_1,\ldots,A_r)$. Let $d_0 := 1-\sum_{i \in [r]}d_i$.
Then
$$
g(A_1,\ldots,A_r) = \sum_{i \in [r]}d_i\log  i.
$$
Define
$$
g(n,r) := \max\left\lbrace g(A_1,\ldots,A_r) : A_i \subseteq [n]\text{ is sum-free for all }i \in [r]\right\rbrace.
$$
Define also $g(A,r)$ to be the maximum of $g(A_1,\ldots,A_r)$ over all tuples of sum-free subsets of $[n]$ such that $\bigcup_{i \in [r]}A_i=A$.

\end{problem}
\begin{remark}\label{maxremark}
It is not hard to see that, for every $n$ and $r$, there is always some tuple $(A_1,\ldots,A_r)$ of sum-free subsets of $[n]$ which is \emph{extremal} (that is, $g(A_1,\ldots,A_r)=g(n,r)$), and $A_i$ is a \emph{maximal} sum-free subset for all $i \in [r]$.
It will be useful to choose such an extremal tuple later, since if we know e.g.~that $A_i$ contains no even element, then by Theorem~\ref{thm-sftypes} we can assume that $A_i=\Odd$.
\end{remark}

\medskip
\noindent
The first step in the proofs of Theorems~\ref{2colours}--\ref{45colours} is to show that these problems are, in a sense, equivalent.

\begin{theorem}\label{problem2}
For all $\eps>0$ and $r \in \mathbb{N}$, there exists $n_0>0$ such that the following holds for all integers $n \geq n_0$.
Let $A \subseteq [n]$.
Then there exists $A' \subseteq A$ with $|A'| \geq |A|-\eps n$ for which
\begin{equation}\label{eq1}
2^{g(A,r)n} \leq f(A,r) \leq 2^{(g(A',r)+\eps) n}.
\end{equation}
Therefore
\begin{equation}\label{eq2}
g(n,r) \leq \frac{\log  f(n,r)}{n} \leq g(n,r) + \eps.
\end{equation}
\end{theorem}

\begin{proof}
We first prove the lower bound (for all $n$ and $r$).
Fix integers $n,r$ and let $A \subseteq [n]$.
Choose a tuple $(A_1,\ldots,A_r)$ of sum-free subsets of $[n]$ whose intersection vector $(d_1,\ldots,d_r)$ is extremal, i.e.~satisfies
\begin{equation}\label{extd}
\sum_{i \in [r]}d_i\log  i = g(A_1,\ldots,A_r) = g(A,r).
\end{equation}

For each $I \in 2^{[r]}\setminus \lbrace \emptyset\rbrace$, define $E_I$ as in the statement of Problem~\ref{prob-equivalent}.
Consider any colouring $\sigma : A \rightarrow [r]$ such that, for each $I \in 2^{[r]}\setminus \lbrace \emptyset\rbrace$ and $x \in E_I$, we have $\sigma(x) \in I$.
Then $\sigma^{-1}(i) \subseteq A_i$ for all $i \in [r]$, so the fact that $A_i$ is sum-free for all $i \in [r]$ implies that $\sigma$ is valid.
Thus the number of such $\sigma$ is a lower bound for the total number of valid colourings, and so
$$
f(A,r) \geq \prod_{I \in 2^{[r]}\setminus \lbrace \emptyset \rbrace}|I|^{|E_I|} = \prod_{i \in [r]}i^{d_in}\stackrel{(\ref{extd})}{=} 2^{g(A,r)n},
$$
as required.

For the remainder of the proof we focus on the upper bound.
		Fix an integer $r$ and let $\eps > 0$.
		We may assume that $\eps \ll 1/r$.
Choose $\eta$ such that $0 < \eta \ll \eps$. 
Apply Theorem~\ref{removal} to obtain $\gamma,n_0 > 0$ such that, for all integers $n \geq n_0$, every $A \subseteq [n]$ which contains at most $\gamma n^2$ Schur triples may be made sum-free by removing at most $\eta n$ elements.
Without loss of generality we may assume that $\gamma \ll \eta$.
Theorem~\ref{container} implies that, by increasing $n_0$ if necessary, for all integers $n \geq n_0$, there exists a family $\mathcal{F} = \mathcal{F}_n$ of containers such that
\begin{itemize}
\item[(i)] every $F \in \mathcal{F}$ contains at most $\gamma n^2$ Schur triples;
\item[(ii)] every sum-free subset of $[n]$ lies in at least one $F \in \mathcal{F}$;
\item[(iii)] $|\mathcal{F}| \leq 2^{\gamma n}$; and
\item[(iv)] $|F| \leq (1/2+\gamma)n$ for all $F \in \mathcal{F}$.
\end{itemize}		
		Without loss of generality, we may assume that $1/n_0 \ll \gamma$ and $\eps \ll 1$.
		We have the hierarchy
		$$
		1/n_0 \ll \gamma \ll \eta \ll \eps \ll 1.
		$$
Given any $n \geq n_0$ and the family $\mathcal{F}$ of containers, for each $F \in \mathcal{F}$, fix a largest sum-free subset $F^*$ of $F$.
Then (i) together with Theorem~\ref{removal} implies that $|F|-\eta n \leq |F^*| \leq |F|$.
		
		Now let $n \geq n_0$ be an integer, and $A \subseteq [n]$ be arbitrary.
		Consider any fixed valid $r$-colouring $\sigma$ of $A$.
		Then $\sigma^{-1}(i)$ is sum-free for all $i \in [r]$.
By (ii), we may choose a tuple $(F_1,\ldots,F_r) \in \mathcal{F}^r$ of containers such that $\sigma^{-1}(i) \subseteq F_i$.
By (i), $F_i$ contains at most $\gamma n^2$ Schur triples for all $i \in [r]$.
For each $i \in [r]$, let us write $F_i^* := (F_i)^*$ for the largest sum-free subset of $F_i$ we fixed earlier.
Then $|F_i^*| \geq |F_i|-\eta n$.
Thus, for each valid colouring $\sigma : A \rightarrow [r]$, we obtain a tuple $(F_1^*,\ldots,F_r^*)$.
Observe that 
\begin{align}\label{eq-removed}
|\sigma^{-1}(i)\setminus F_i^*| \leq \eta n,
\end{align}
 but $F_i^*$ may contain many elements which do not lie in $\sigma^{-1}(i)$.

We now claim that the following procedure generates every valid colouring $\sigma$ of $A$, and therefore the number of choices in this procedure is an upper bound on $f(A,r)$. Each choice will generate a colouring $\tau$:
		\begin{enumerate}
			\item For all $i \in [r]$, choose a container $G_i \in \mathcal{F}$, and let $G_i^*$ be the largest sum-free subset of $G_i$ we fixed earlier.
			\item For each $I \in 2^{[r]}\setminus \lbrace \emptyset\rbrace$, let
			$$
			E'_I := \left(\bigcap_{i \in I}G_i^*\setminus \bigcup_{j\notin I}G_j^* \right) \cap A.
			$$
			Let also $D'_i := \bigcup_{I \subseteq [r]:|I|=i}E'_I$ for each $i \in [r]$.
			So $D'_i$ is the set of those elements in $A$ which lie in exactly $i$ of the $G_j^*$. Let $d'_i := |D'_i|/n$ for all $i \in [r]$.
			\item For each $I \in 2^{[r]}\setminus \lbrace \emptyset\rbrace$ and $x \in E'_I$, choose $i \in I$ and set $\tau(x):=i$.
			\item For each uncoloured $y \in A$, let $\tau(y) \in [r]$ be arbitrary.
		\end{enumerate}
We need to show that there is a choice in (1)--(4) which will yield $\tau=\sigma$.
In (1), for each $i \in [r]$, (ii) and the fact that $\sigma$ is valid implies that we can choose $G_i := F_i \in \mathcal{F}$ such that $\sigma^{-1}(i) \subseteq F_i$.
%Choose the $F_i'$ in (2) arbitrarily. Note that, by Theorem~\ref{removal}, $|\sigma^{-1}(i) \setminus F_i'| \leq \eta n$.
Note that $G_i^*=F_i^*$ for all $i \in [r]$.
The choice in (2) is fixed by our choices in (1).
In (3), by construction,
for every $x \in (\bigcup_{i \in [r]}F_i^*)\cap A$, we have that $x \in E'_I$ for some $I \ni \sigma(x)$.
Thus for every $x \in (\bigcup_{i \in [r]}F_i^*)\cap A$ we can choose $\tau(x) := \sigma(x)$.
In (4) we are free to colour the uncoloured elements of $A$ with $\sigma$.
Since $\sigma$ was an arbitrary valid colouring of $A$, we have proved the claim.

Thus it remains to count the number of colourings generated by (1)--(4).
Given a tuple $(G_1,\ldots,G_r) \in \mathcal{F}^r$ of containers, let $\mathcal{C}$ be the set of colourings $\tau : A \rightarrow [r]$  generated by it, i.e.~the set of $\tau$ which arise from the procedure after fixing the choice $(G_1,\ldots,G_r)$ in~(1). Observe that $(G_1,\ldots,G_r)$ gives rise to a unique tuple $(G_1^*,\ldots,G_r^*)$. Then, since the only choices are in (3) and (4), we have
$$
|\mathcal{C}| \leq \prod_{j \in [r]}j^{|D'_j|} \cdot r^{\left|A\setminus \bigcup_{i \in [r]}G_i^*\right|} \stackrel{\eqref{eq-removed}}{\leq} \prod_{j \in [r]}j^{d'_jn} \cdot r^{r\eta n}.
$$
Taking logarithms, we have that
\begin{equation}\label{logeq}
\frac{\log  |\mathcal{C}|}{n} \leq \sum_{j \in [r]}d'_j \log  j + \sqrt{\eta} = g(G_1^*\cap A,\ldots,G_r^* \cap A) + \sqrt{\eta} \leq g(A',r) + \sqrt{\eta},
\end{equation}
where $A' := A \cap \bigcup_{i \in [r]}G_i^*$.
So $|A'| \geq |A|-r\eta n \geq |A|-\eps n$.
But, by (iii), the number of choices of $(G_1,\ldots,G_r) \in \mathcal{F}^r$ is at most $|\mathcal{F}|^r \leq 2^{r\gamma n}$, so
$$
f(A,r) \leq 2^{r\gamma n} \cdot 2^{g(A',r)n} \cdot 2^{\sqrt{\eta}n} \leq 2^{(g(A',r)+2\sqrt{\eta})n} \leq 2^{(g(A',r)+\eps)n},
$$
completing the proof of the upper bound.
The second assertion is an obvious consequence of the first.
\end{proof}

The discussion about lower bounds after the statement of Theorem~\ref{45colours} amounts to the following inequalities:
For all integers $n \geq r \geq 2$, we have
\begin{equation}\label{gbound1}
f(n,r) \geq g(n,r) \geq g(\underbrace{\Odd,\ldots,\Odd}_{r}) = g(\underbrace{I_2,\ldots,I_2}_{r}) = \frac{\lceil n/2\rceil}{n} \cdot \log  r; \quad\text{and}
\end{equation}
\begin{equation}\label{gbound2}
f(n,r) \geq g(n,r) \geq g(\underbrace{I_2,\ldots,I_2}_{\lfloor r/2\rfloor},\underbrace{\Odd,\ldots,\Odd}_{\lceil r/2\rceil}) \geq \frac{\lfloor n/4\rfloor}{n} \log \left( r\left\lfloor \frac{r^2}{4}\right\rfloor \right).
\end{equation}
Figure~\ref{solsfig2} shows these three constructions in the case when $r=4$, when they each give rise to roughly the same lower bound. Theorem~\ref{45colours}) implies that each one is in fact an approximate optimal solution of Problem~1.

Observe the following easy correspondence between feasible solutions of Problems~1 and~2.
Given a feasible solution $A \subseteq [n]$ of Problem~1 and a valid $r$-colouring $\sigma$ of $A$, we have that $(\sigma^{-1}(1),\ldots,\sigma^{-1}(r))$ is a feasible solution of Problem~2.
Given a feasible solution $(A_1,\ldots,A_r) \subseteq [n]^r$ of Problem~\ref{prob-equivalent}, we have that $A_1 \cup \ldots \cup A_r$ is a feasible solution of Problem~\ref{prob-main}.

Theorem~\ref{problem2} is essentially an analogue of the main result of~\cite{psy}.
%Indeed, the main result of~\cite{psy} is that asymptotically determining the logarithm of the parameter $F(n,s,k)$ is equivalent to solving a certain finite optimisation problem, whose maximum $Q(s,k)$ is independent of $n$.
Informally speaking, determining $g(n,r)$ involves layering $r$ sum-free subsets of $[n]$ so that an appropriately weighted overlap is as large as possible, whereas determining $h(r,k)$ involves layering $r$ finite $K_k$-free graphs so that their weighted overlap is as large as possible.
Importantly and unfortunately, $g(n,r)$ does of course depend on $n$.
However, the cases in which $F(n,r,k)$ has been determined (when $r$ is small) give us some valuable intuition for determining $g(n,r)$ (and hence approximately determining $f(n,r)$):
namely that for an extremal tuple $(A_1,\ldots,A_r)$ of sum-free sets, each $A_i$ should perhaps be a \emph{largest} sum-free set: either $\Odd$ or $I_2$.
Unlike in the case of graphs, our ground set $[n]$ comes with a fixed labelling.
So there is only one way to layer, say, $\Odd$ and $I_2$, whereas there are many ways to layer any two $r$-vertex graphs $G$ and $H$.
%Here it is instructive to compare Figures~\ref{solsfig} and~\ref{solsfig2}.

\begin{center}

\begin{figure}[]
\scalebox{0.7}{
\begin{tikzpicture}[scale=1]
\begin{scope}
\draw[line width = 0.5mm] (0,0)--(6,0);
\draw[line width = 0.5mm] (0,-0.1)--(0,0.1);
\draw[line width = 0.5mm] (3.01,-0.1)--(3.01,0.1);
\draw[line width = 0.5mm] (5.99,-0.1)--(5.99,0.1);

\draw[line width = 2mm, color=Magenta] (3,0.3)--(6,0.3);
\draw[line width = 2mm, color=Green] (3,0.6)--(6,0.6);

\foreach \x in {0,0.1,...,6.1} \draw[line width = 0.5mm, color=Blue] (\x,0.8)--(\x,1.0);
\foreach \x in {0,0.1,...,6.1} \draw[line width = 0.5mm, color=Red] (\x,1.1)--(\x,1.3);

\node at (0,-0.8) [draw=none,
label={\large $1$}] (){};
\node at (3,-1) [draw=none,
label={\large $\left\lfloor\frac{n}{2}\right\rfloor$}] (){};
\node at (6,-0.8) [draw=none,
label={\large $n$}] (){};
\end{scope}

%%%%%%%%%%%%%%

\begin{scope}[xshift=-8cm]
\draw[line width = 0.5mm] (0,0)--(6,0);
\draw[line width = 0.5mm] (0,-0.1)--(0,0.1);
\draw[line width = 0.5mm] (3.01,-0.1)--(3.01,0.1);
\draw[line width = 0.5mm] (5.99,-0.1)--(5.99,0.1);

\draw[line width = 2mm, color=Magenta] (3,0.3)--(6,0.3);
\draw[line width = 2mm, color=Green] (3,0.6)--(6,0.6);
\draw[line width = 2mm, color=Blue] (3,0.9)--(6,0.9);
\draw[line width = 2mm, color=Red] (3,1.2)--(6,1.2);

\node at (0,-0.8) [draw=none,
label={\large $1$}] (){};
\node at (3,-1) [draw=none,
label={\large $\left\lfloor\frac{n}{2}\right\rfloor$}] (){};
\node at (6,-0.8) [draw=none,
label={\large $n$}] (){};
\end{scope}

%%%%%%%%%%%%%%%%%%%%

\begin{scope}[xshift=-16cm]
\draw[line width = 0.5mm] (0,0)--(6,0);
\draw[line width = 0.5mm] (0,-0.1)--(0,0.1);
\draw[line width = 0.5mm] (3.01,-0.1)--(3.01,0.1);
\draw[line width = 0.5mm] (5.99,-0.1)--(5.99,0.1);

\foreach \x in {0,0.1,...,6.1} \draw[line width = 0.5mm, color=Magenta] (\x,0.2)--(\x,0.4);
\foreach \x in {0,0.1,...,6.1} \draw[line width = 0.5mm, color=Green] (\x,0.5)--(\x,0.7);
\foreach \x in {0,0.1,...,6.1} \draw[line width = 0.5mm, color=Blue] (\x,0.8)--(\x,1.0);
\foreach \x in {0,0.1,...,6.1} \draw[line width = 0.5mm, color=Red] (\x,1.1)--(\x,1.3);

\node at (0,-0.8) [draw=none,
label={\large $1$}] (){};
\node at (3,-1) [draw=none,
label={\large $\left\lfloor\frac{n}{2}\right\rfloor$}] (){};
\node at (6,-0.8) [draw=none,
label={\large $n$}] (){};
\end{scope}
\end{tikzpicture}
\caption{Three approximate solutions to Problem~2 for $r=4$.}
\label{solsfig2}
}
\end{figure}
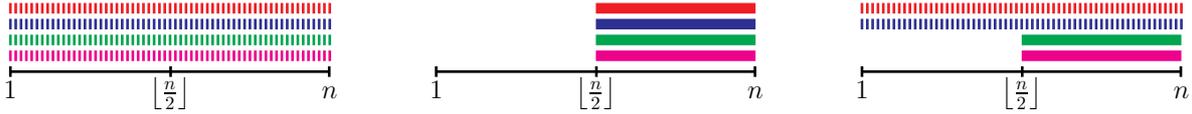
\end{center}

%%%%%%%%%%%%%%%%%%%%%%%%%%%%%%%%%%%%%%%%%%%%%%%%%%%%%%%%%%%%%%%%%%%%%%%%%%%
\section{The proof of Theorem~\ref{3colours}}\label{sec:3}

Given Theorem~\ref{problem2}, it is now a fairly simple task to obtain stability in the case when $r=2,3$.
Indeed, it suffices to prove stability for Problem 2.

\begin{lemma}\label{problem2stab23}
Let $r \in \lbrace 2,3\rbrace$.
For all $\eps > 0$, there exists $n_0>0$ such that the following holds for all integers $n \geq n_0$.
Let $A \subseteq [n]$ be such that $g(A,r) \geq g(n,r) - \eps/(50r)$.
Then either
$|A\bigtriangleup \Odd| \leq \eps n$ or $|A\bigtriangleup I_2| \leq \eps n$.
\end{lemma}

\begin{proof}
Let $\eps > 0$ and assume without loss of generality that $\eps < 1/100$.
Let also $n$ be a sufficiently large integer compared to $\eps$.
Choose a tuple $(A_1,\ldots,A_r)$ of sum-free subsets of $[n]$ such that $g(A_1,\ldots,A_r)=g(A,r)$.
Let $(d_1,\ldots,d_r)$ be its intersection vector.
Recall from~(\ref{gbound1}) that $g(n,r) \geq (1/2)\cdot \log  r$.
So $g(A,r) \geq (1/2)\cdot \log  r-\eps/(50r)$.
We will need the following claim.

\begin{claim}
It suffices to show that $d_r \geq 1/2 - \eps/(3r)$.
\end{claim}

\begin{claimproof}
Let $A' := A_1 \cap \ldots \cap A_r$.
Assume that $|A'|=d_r n \geq (1/2-\eps/(3r))n$.
Clearly $A'$ is a sum-free subset of $[n]$.
Suppose that $A'$ contains at least one even element.
Then Theorem~\ref{thm-sftypes} implies that $A'$ is a sum-free set of type (c), so $\min(A') \geq |A'| \geq (1/2-\eps/(3r))n$, and so $|A' \bigtriangleup I_2| \leq 2\eps n/(3r)$.
Otherwise (if every element of $A'$ is odd) we have $|A' \bigtriangleup \Odd| \leq \eps n/(3r)$.
Finally, since every sum-free subset of $[n]$ has size at most $\lceil n/2\rceil$, we have
$$
|A \bigtriangleup A'| \leq \sum_{i \in [r]}|A_i \bigtriangleup A'| \leq r(\lceil n/2\rceil - d_rn) \leq \eps n/3 + r.
$$
So, by the triangle inequality, either $|A \bigtriangleup I_2| \leq\eps n$ or $|A \bigtriangleup \Odd| \leq \eps n$.
\end{claimproof}

\medskip
\noindent
First consider the case when $r=2$.
Then
$$
\frac{1}{2}-\frac{\eps}{50r} \leq g(A,2) = g(A_1,A_2) = d_2,
$$
as required.
Now let $r=3$.
Then
\begin{align}\label{eq-2,3c-d2d3}
\frac{1}{2} \cdot \log  3 - \frac{\eps}{50r} \leq g(A,3) = g(A_1,A_2) = d_2 + d_3\cdot\log  3.
\end{align}
Recall that each $A_i$ is sum-free with size at most $\lceil n/2\rceil$, so
$$
d_1+2d_2+3d_3 = \frac{1}{n}\left(|A_1|+|A_2|+|A_3|\right) \leq \frac{3}{n}\left\lceil \frac{n}{2}\right\rceil \leq \frac{3}{2} + \frac{\eps}{100r},
$$
which implies that 
$
d_2+ 3d_3/2 \leq 3/4 + \eps/(200r)
$.
Therefore, if $d_3 < 1/2-\eps/(3r)$, we have
\begin{align*}
d_2+d_3 \cdot \log  3 &= d_2+\frac{3}{2}d_3+\left(\log  3-\frac{3}{2}\right)\cdot d_3  <\frac{3}{4} + \frac{\ep}{200r}+\left(\log  3-\frac{3}{2}\right)\cdot\left(\frac{1}{2}-\frac{\ep}{3r}\right) \\
&= \frac{1}{2}\log  3-\left(\frac{\log  3}{3}-\frac{1}{200}-\frac{1}{2}\right)\cdot\frac{\ep}{r}<\frac{1}{2}\log  3-\frac{\ep}{50r},
\end{align*}
a contradiction to~\eqref{eq-2,3c-d2d3}.
\end{proof}

\begin{corollary}\label{23stability}
	Let $r \in \lbrace 2,3\rbrace$. For all $\eps>0$, there exists $n_0>0$ such that the following holds.
	Let $n > n_0$ be an integer and let $A \subseteq [n]$ be such that $f(A,r) \geq f(n,r)\cdot 2^{-\eps n/(200r)}$.
	Then either $|A\bigtriangleup \Odd| \leq \eps n$ or $|A\bigtriangleup I_2| \leq \eps n$.
\end{corollary}

\begin{proof}
Choose $n_0>0$ sufficiently large so that the conclusion of Theorem~\ref{problem2} holds with parameter $\eps/(200r)$ and the conclusion of Lemma~\ref{problem2stab23} holds with parameter $\eps/2$.
Let $n \geq n_0$ be an integer and let $A \subseteq [n]$ be such that $f(A,r) \geq f(n,r) \cdot 2^{-\eps n/(200r)}$.
From Theorem~\ref{problem2} obtain $A' \subseteq A$ with $|A'| \geq |A|-\eps n/2$.
Then
$$
g(A',r) \stackrel{(\ref{eq1})}{\geq} \frac{\log  f(A,r)}{n}-\frac{\eps}{200r} \geq \frac{\log  f(n,r)}{n}-\frac{\eps}{100r} \stackrel{(\ref{eq2})}{\geq} g(n,r)-\frac{\eps}{100r}.
$$
Lemma~\ref{problem2stab23} implies that either $|A' \bigtriangleup \Odd| \leq \eps n/2$ or $|A' \bigtriangleup I_2| \leq \eps n/2$. The result follows.
\end{proof}

Note that, in both cases $r \in \lbrace 2,3\rbrace$, the proof proceeds by solving a linear program in variables $d_1,\ldots,d_r$.
The very same linear program (approximately) yields $F(n,r,3)$ for $r \in \lbrace 2,3\rbrace$ in the proof of the main result in~\cite{abks}, where the variables $d_i$ correspond to densities of overlapping triangle-free graphs.

%%%%%%%%%%%%%%%%%%%%%%%%%%%%%%%%%%%%%%%%%%%%%%%%%%%%%%%%%%%%%%%%%%%%%%%%%%%

%%%%%%%%%%%%%%%%%%%%%%%%%%%%%%%%%%%%%%%%%%%%%%%%%%%%%%%%%%%%%%%%%%%%%%%%%%%

\section{The proof of Theorem~\ref{2colours}}\label{sec:2}

In this section we use Corollary~\ref{23stability} to prove Theorem~\ref{2colours}.
Before starting the proof, we need the following useful notion.

\begin{definition}
Given sets $S, B\subseteq [n]$, define the \emph{link graph of $B$ generated by $S$}, denoted $L_S[B]$, as follows.
We set $V(L_S[B]):=B$, and given $x,y \in B$, we have $xy \in E(L_S[B])$ if and only if there is some $z \in S$ such that $\lbrace x,y,z\rbrace$ is a Schur triple, and $x,y,z$ are distinct.
Note that $L_S[B]$ is a simple graph (i.e.~it does not contain loops).
If $S = \lbrace v \rbrace$, then we use the shorthand $L_v[B] := L_{\lbrace v \rbrace}[B]$.
\end{definition}

This notion is useful since each edge in a link graph represents some restriction of valid colourings, and so a large set of independent edges in a link graph limits the number of valid colourings.

\begin{lemma}\label{matching}
Let $n,r \in \mathbb{N}$ and let $A \subseteq [n]$.
Let $x \in A$ and suppose that the link graph $L_x[A\setminus \lbrace x \rbrace]$ contains a matching $M_x$ of size $m$. Then
$$
f(A,r) \leq (r^2-1)^m \cdot r^{|A|-2m}.
$$
\end{lemma}

\begin{proof}
	We will bound $f(A,r)$ by re-constructing valid colourings using the following procedure, in which every choice yields a colouring $\tau : A \rightarrow [r]$ (which may or may not be valid).
	
	\begin{itemize}
	\item[(1)] Let $\tau(x) \in [r]$ be arbitrary.
	\item[(2)] For each edge $uv$ in $M_x$, choose $(\tau(u),\tau(v)) \in [r]^2 \setminus \lbrace (\tau(x),\tau(x))\rbrace$.
	\item[(3)] For each uncoloured $y \in A$, let $\tau(y) \in [r]$ be arbitrary.
	\end{itemize}
	
To see that the procedure generates every valid $r$-colouring of $A$, we just need to check that every valid $r$-colouring $\sigma$ has the property in Step~2 (since all other choices were arbitrary). That is, for all $uv \in M_x$, $\sigma$ does not assign $u,v,x$ the same colour.
But this is clear since $\lbrace u,v,x\rbrace$ is a Schur triple. 

Therefore the number of colourings generated by the procedure is an upper bound for $f(A,r)$.
Thus, using the fact that $|A\setminus (\{x\} \cup V(M_x))| = |A|-2m-1$, we have
$$
f(A,r) \leq r \cdot (r^2-1)^m \cdot r^{|A|-2m-1},
$$ 
as required.
\end{proof}

We are now ready to prove Theorem~\ref{2colours}.

\medskip
\noindent
\emph{Proof of Theorem~\ref{2colours}.}
First recall that $2^{\lceil n/2\rceil}$ is a lower bound for $f(n,2)$ for all positive integers $n$, since we can always find a sum-free subset of $[n]$ of size $\lceil n/2\rceil$, namely $\Odd$ or $I_2$, and colour it arbitrarily. 
Let $0 < \eps < 1/200$ and apply Corollary~\ref{23stability} to obtain $n_0$.
Now fix an integer $n > n_0$, and let $A \subseteq [n]$ be such that
$$
f(A,2)=f(n,2) \geq 2^{\lceil n/2\rceil}.
$$
By Corollary~\ref{23stability}, we have that either
$|A \bigtriangleup \Odd| \leq \eps n$; or
$|A \bigtriangleup I_2| \leq \eps n$.
Thus $|A|\le \ceil{n/2}+\ep n$. We will use the following claim throughout the proof.
\begin{claim}\label{cl-matching}
	For all $x\in A$, every matching in $L_x[A\setminus\{x\}]$ has size less than $3\ep n$.
\end{claim}

Indeed, if not, then Lemma~\ref{matching} implies that
\begin{eqnarray*}%\label{eq-2colmatching}
f(A,2)\le 3^{3\ep n}\cdot 2^{|A|-6\ep n}\le  3^{3\ep n}\cdot 2^{\ceil{\frac{n}{2}}+\ep n-6\ep n}\le 2^{\ceil{\frac{n}{2}}-\frac{\ep n}{5}},
\end{eqnarray*} 
a contradiction.

To find large matchings in link graphs in the next two claims, we will use the fact that a graph $G$ with $e$ edges and maximum degree $\Delta$ contains a matching of size at least $e/(\Delta+1)$.
This is an immediate consequence of Vizing's theorem on edge-colourings.

%We will show in Lemma~\ref{lem-odd} that if $A$ contains an even element, then it has exponentially fewer sum-free colourings.

%%%%%%%%%%%%%%%%%%%%%%%%%%%%%%%%%%%%%%%%%%%%%%%%%%%%%
%%%%%%%%%%%%%%%%%%%%%%%%%%%%%%%%%%%%%%%%%%%%%%%%%%%%%

\noindent
\begin{case} $|A \bigtriangleup \Odd| \leq \eps n$.
\end{case}

\noindent
In this case, we will prove that $A = \Odd$.
It suffices to show that $A\cap E=\emptyset$, since for any such set we have $f(A,2)=2^{|A|}$, and thus $A=\Odd$ is clearly the unique extremal subset among such sets. So suppose that there is some $x \in A \cap E$. The following claim together with Claim~\ref{cl-matching} will complete the proof of this case.

\begin{claim}\label{lem-odd-large-matching}
	$L_x[A\setminus\{x\}]$ contains a matching $M_x$ of size at least $9\ep n$.
\end{claim}
\begin{claimproof}
Since $x$ is even, $A\cap\Odd\subseteq A\setminus\{x\}$, and thus $L_x[A\cap\Odd]\subseteq L_x[A\setminus\{x\}]$. Also, since $|A \bigtriangleup \Odd| \leq \eps n$, it suffices to show that $L_x[\Odd]$ contains a matching of size at least $10\eps n$. %Define $m$ to be the size of a maximum matching in $L_x[\Odd]$. 
First, we will assume that $x\le n-60\ep n$.  In $L_x[\Odd]$, every odd number in the set $[x]\cup [n-x+1,n]$ has degree one, and all the other odd numbers have degree two. Hence, there are exactly $x$ integers of degree one, and the remaining $\lceil (n-2x)/2\rceil$ integers have degree two. Therefore $e(L_x[\Odd]) \geq \frac{1}{2}(x+(n-2x))$.
By Vizing's theorem, $L_x[\Odd]$ contains a matching of size at least
\begin{eqnarray*}
\frac{e\left(L_x[\Odd]\right)}{\Delta\left(L_x[\Odd]\right)+1}\ge  \frac{n-x}{6}\ge 10\ep n,
\end{eqnarray*}
a contradiction.
Now, if $x>n-60\ep n$, then $M=\{\{1,x-1\},\{3,x-3\},\ldots,\{x/2-1,x-(x/2-1)\}\}\subseteq L_x[\Odd]$ is a matching of size $\floor{x/4} \geq n/4-20\eps n \geq 10\eps n$.
\end{claimproof}
%%%%%%%%%%%%%%%%%%%%%%%%%%%%%%%%%%%%%%%%%%%%%%%%%%%%%
%%%%%%%%%%%%%%%%%%%%%%%%%%%%%%%%%%%%%%%%%%%%%%%%%%%%%

\begin{case} $|A \bigtriangleup I_2| \leq \eps n$.
\end{case}

\noindent
Now $A$ has few elements from $I_1$. We first show that the smallest element in $A$ cannot be far from $n/2$. For all $i \in \mathbb{N}$, denote by $\de_i$ the $i^{\rm th}$ smallest element in $A$.
We may suppose that $\delta_1 \in I_1$ since otherwise $A \subseteq I_2$ is sum-free and we are done.

\begin{claim}\label{de1}
$\de_1 > n/2-12\eps n$.
\end{claim}

\begin{claimproof}
Suppose not.
	Similar to Case~1, since $\de_1\in I_1$, we have that $A\cap I_2\subseteq A\setminus\{\de_1\}$. Therefore, it suffices to show that the link graph $L_{\de_1}[A\cap I_2]$ contains a matching of size at least $3\ep n$ which contradicts Claim~\ref{cl-matching}. Since $|A \bigtriangleup I_2| \leq \eps n$, we only need to show that $L_{\de_1}[I_2]$ contains a matching of size at least $4\ep n$. %Define $m$ to be the size of the maximum matching in $L_{\de_1}[I_2]$.  
	In $L_{\de_1}[I_2]$, every element in $[\floor{n/2}+1,\floor{n/2}+\de_1]\cup [n-\de_1+1,n]$ has degree one, and all the other $\ceil{n/2}-2\de_1$ elements have degree two. Therefore $e(L_{\delta_1}[I_2]) \geq \frac{1}{2}(2\delta_1+2(\lceil n/2\rceil-2\delta_1)) = \lceil n/2\rceil-\delta_1$. Therefore, by Vizing's theorem, $L_{\delta_1}[I_2]$ contains a matching of size at least
	\begin{eqnarray*}
		\frac{e\left(L_{\de_1}[I_2]\right)}{\Delta\left(L_{\de_1}[I_2]\right)+1}\ge 
		\frac{\ceil{\frac{n}{2}}-\de_1}{3}\ge \frac{n}{6}-\frac{\de_1}{3}\ge 4\ep n,
	\end{eqnarray*}
a contradiction.
\end{claimproof}

\medskip
\noindent
Define $k$ such that $A \cap I_1 = \lbrace \delta_1,\ldots,\delta_k\rbrace$.
Then $B := 2\cdot (A \cap I_1)$ is a subset of $I_2$ by Claim~\ref{de1}.
There are two ways of colouring each of the $k$ pairs $\lbrace \delta_i,2\delta_i\rbrace$, so
$$
f(A,r) \leq 2^k \cdot 2^{(A \cap I_2) \setminus B} \leq 2^k \cdot 2^{|I_2|-k} = 2^{\lceil n/2\rceil}
$$
with equality in the second inequality if and only if $|(A \cap I_2)\setminus B| = |I_2|-k$, i.e. if and only if $I_2\setminus A \subseteq B$.

Suppose now that $\delta_1+\delta_2 \leq n$.
Then $\delta_1 \in I_1$ (but $\delta_2$ may or may not be).
Since $\delta_1+\delta_2 \notin B$, we have $\delta_1+\delta_2 \in A$. 
Let $C := \lbrace \delta_1,\delta_2,2\delta_1,\delta_1+\delta_2,2\delta_2\rbrace$.
Then $\lbrace \delta_1,\delta_2,\delta_1+\delta_2\rbrace \subseteq A \cap C \subseteq C$, and one can easily check that there are at most six valid colourings of $A \cap C$.

Suppose first that $\delta_2 \in I_1$. Then $I_2 \cap C = \lbrace 2\delta_1,\delta_1+\delta_2,2\delta_2\rbrace$.
So
$$
f(A,r) \leq 2^{k-2} \cdot 2^{|I_2\setminus C|-(k-2)} \cdot 6 = 6 \cdot 2^{|I_2|-3} = \frac{3}{4} \cdot 2^{\lceil n/2\rceil},
$$
a contradiction.
Suppose instead that $\delta_2 \in I_2$. Then $k=1$ and $I_2 \cap C = \lbrace \delta_2,2\delta_1,\delta_1+\delta_2\rbrace$, and
$$
f(A,r) \leq 2^{|I_2\setminus C|} \cdot 6 = 6 \cdot 2^{|I_2|-3} = \frac{3}{4}\cdot 2^{\lceil n/2\rceil}. 
$$
Thus $\delta_1+\delta_2 \geq n+1$.

So certainly $|A \cap I_1| \leq 1$, and $I_2 \setminus A \subseteq \lbrace 2\delta_1\rbrace$, but we can say more.
If $A \cap I_1=\emptyset$ then $A = I_2$ and we are done.
So we may assume that $A\cap I_1 = \lbrace \delta_1\rbrace$.
Then either $\delta_2=\lfloor n/2\rfloor+1$; or $2\delta_1=\lfloor n/2\rfloor+1$ and $\delta_2=\lfloor n/2\rfloor+2$.
In either case,
$$
n+1 \leq \delta_1+\delta_2 \leq 2\lfloor n/2\rfloor +1.
$$
So $n$ is even and we have equality if and only if $(\delta_1,\delta_2)=(n/2,n/2+1)$.
Since $I_2\setminus A \subseteq \lbrace n\rbrace$, there are two candidates for extremal sets: $[n/2,n-1]$ and $[n/2,n]$.
It is easy to see that both have $2^{n/2}$ valid $2$-colourings.
This completes the proof of Theorem~\ref{2colours}.
\hfill$\square$

%%%%%%%%%%%%%%%%%%%%%%%%%%%%%%%%%%%%%%%%%%%%%%%%%%%%%%%%%%%%%%%%%%%%%%%%%%%
\section{The proof of Theorem~\ref{45colours}}\label{sec:45}

To prove Theorem~\ref{45colours}, by Theorem~\ref{problem2}, it suffices to determine $g(n,r)$ asymptotically.
Our aim is to show that
$$
g(n,4) = 1 + o(1) \quad\text{and}\quad g(n,5) = \frac{1}{4}\log  30 + o(1).
$$
Recall from the proof of Lemma~\ref{problem2stab23} that we reduced the problem of asymptotically determining $g(n,3)$ to solving a \emph{linear} program. Indeed, we let $(A_1,A_2,A_3)$ be a tuple of sum-free subsets with intersection vector $(d_1,d_2,d_3)$, and found a reduction of Problem~2 into a linear program in variables $d_1,d_2,d_3$.
Our task was then to maximise $\sum_{i \in [3]}d_i\log  i$ subject to $d_1,d_2,d_3 \geq 0$ and $\sum_{i \in [3]}id_i \leq \frac{3}{n}\left\lceil\frac{n}{2}\right\rceil$.

So to prove Theorem~\ref{45colours}, we will again reduce Problem~2 to a linear program in variables $d_1,\ldots,d_r$ and an additional slack variable $a$ (defined below) for $r \in \lbrace 4,5\rbrace$.

%We must find a set of linear constraints $\lbrace a_{i1}d_1 + \ldots + a_{ir}d_r \leq b_i : i \in C_{j_r} \rbrace$ which an optimal solution $(d_1,\ldots,d_r)$ to Problem~2 must satisfy, and such that the maximum of $\sum_{i \in [r]}d_i\log i$ subject to these constraints is at most $1+o(1)$ when $r=4$ and at most $\frac{1}{4}\log  30 + o(1)$ when $r=5$.

Given $n \in \mathbb{N}$, let $A_1,\ldots, A_r\in [n]$ be maximal sum-free sets. Throughout the rest of this section, define $D_1,\ldots,D_r,d_1,\ldots,d_r$ as in Problem~2. Further, we let $C$ be the subset of $[r]$ such that $A_i$ is a type~(a) set for every $i\in C$, and  a type~(b) or~(c) set, for every $i\in [r]\setminus C$. We define 
\begin{align}\label{eq-defab}
a&:=\frac{1}{n}\sum_{i\in [r]\setminus C}\left(\ceil{\frac{n}{2}}-|A_i|\right)=\frac{(r-|C|)}{n}\cdot \ceil{\frac{n}{2}}-\frac{1}{n}\sum_{i \in [r]\setminus C}|A_i|.
\end{align}
An important observation is that, if $A$ is a maximal sum-free set of type (b), then $A$ is precisely $\Odd$, the set of odd integers in $[n]$.

\begin{definition}
Let $r,n \in \mathbb{N}$ and $\eps>0$. %Let $A_1,\ldots,A_r$ be maximal sum-free sets in $[n]$ and define $D_1,\ldots,D_r,d_1,\ldots,d_r$ as in Problem~2, and $a$ as in~\ref{eq-defab}.
Suppose that
$$
\alpha_{i1}d_1 +\ldots + \alpha_{ir}d_r + \alpha_ia \leq \beta_i\quad\text{for all}\quad i \leq N(r),
$$
where $N(r)$ is a positive integer depending only on $r$; and $\alpha_{ij},\beta_{ij},\alpha_i \in \mathbb{R}$ for all $i \leq N(r)$ and $j \leq r$.
We say that the set $\mathcal{C}$ whose members are these inequalities is a \emph{family of constraints for $A_1,\ldots,A_r$}.
Further, $\mathcal{C}$ is \emph{$(\eps,r)$-sufficient} if
$$
\max \sum_{i \in [r]}d_i\log i \quad\text{subject to}\quad \mathcal{C} \quad\text{is at most}\quad \max \left\lbrace \frac{1}{2}\log r, \frac{1}{4}\log(r\lfloor r^2/4\rfloor)\right\rbrace + \eps.
$$
\end{definition}

As an example, when $r=3$, we showed that, for every $n \in \mathbb{N}$ and sum-free subsets $A_1,A_2,A_3$ of $[n]$, the family
$$
\left\lbrace d_1 \geq 0; d_2 \geq 0; d_3 \geq 0; d_1+2d_2+3d_3 \leq \frac{3}{n}\lceil n/2\rceil \right\rbrace
$$
of constraints is $(0,3)$-sufficient. For every $\eps>0$, when $n$ is sufficiently large, the family obtained from it by replacing the final inequality with $d_1+2d_2+3d_3 \leq 3/2 + \eps/2$ is still a family of constraints for $A_1,A_2,A_3$, is independent of $n$, and is $(\eps,3)$-sufficient.

The proof of Theorem~\ref{45colours} follows from the next lemma.

\begin{lemma}\label{45colours2}
	Let $r \in \lbrace 4,5\rbrace$. For all $\ep>0$, there exists an $n_0>0$ such that for all integers $n\ge n_0$, every choice of maximal sum-free subsets $A_1,\ldots,A_r$ of $[n]$ has a family of $(\eps,r)$-sufficient constraints.
\end{lemma}

Indeed, suppose that the lemma holds.
The construction after the statement of Theorem~\ref{45colours} shows that, whenever $n$ is a sufficiently large integer, we have $f(n,4) \geq 16^{\lfloor n/4\rfloor} > 2^{(1-\eps)n}$.
Thus it suffices to find $n_0>0$ such that $f(n,4) \leq 2^{(1+\eps)n}$ whenever $n \geq n_0$ is an integer.
Choose $n_0$ so that the conclusions of Theorem~\ref{problem2} (with $r=4$) and Lemma~\ref{45colours2} hold for parameter $\eps/2$ and all $n \geq n_0$.
Now let $n \geq n_0$ be an arbitrary integer. 
By Remark~\ref{maxremark}, there are maximal sum-free subsets $A_1,\ldots,A_4$ of $[n]$ such that $g(A_1,\ldots,A_4) = g(n,4)$.
By Lemma~\ref{45colours2}, $A_1,\ldots,A_4$ has a family of $(\eps/2,4)$-sufficient constraints. Thus $g(n,4) = g(A_1,\ldots,A_4) = \sum_{i \in [4]}d_i \log i \leq 1 + \eps/2$. 
Then
$$
f(n,4) \stackrel{(\ref{eq2})}{\leq} 2^{(g(n,4)+\eps/2)n} \leq 2^{(1+\eps)n},
$$
as required.
The case $r=5$ is almost identical.

\medskip
\noindent
For each $r=4,5$, we split the proof of Lemma~\ref{45colours2} into cases depending on the structure of $A_1,\ldots,A_r$ (obtained from Theorem~\ref{thm-sftypes}).
Then, in each case, we find a family of constraints which is $(\eps,r)$-sufficient.
Given a family $\mathcal{C}$ of inequalities, we must
\begin{enumerate}
\item show that it is a family of constraints for $A_1,\ldots,A_r$, i.e.~that each inequality holds; then
\item show that it is $(\eps,r)$-sufficient, i.e.~consider the linear program $\max \sum_{i \in [r]}d_i \log i$ subject to $\mathcal{C}$, and show that its optimal solution is at most the required value.
\end{enumerate}
Since it is only a serious of tedious calculations, we defer the details of~(2) to the appendix, and limit ourselves to some remarks here.

\subsection{Achieving~(2): Solving linear programs}

Since there are many cases (depending on the structure of $A_1,\ldots,A_r$), and sometimes rather a lot of inequalities in each family of constraints $\mathcal{C}$, where possible, we use Mathematica to solve the resulting linear program $\max_{i \in [r]}d_i\log i$ subject to $\mathcal{C}$.
Suppose that $r=4$ (when $r=5$ the situation is similar).
Given $A_1,\ldots,A_4$, an $(\eps,4)$-sufficient family $\mathcal{C}$ is such that $\sum_{i \in [4]}d_i \log i \leq 1 + \eps$.

There are two cases, depending on whether $A_1,\ldots,A_4$ is close to extremal or not.
Suppose that there is some specific value of $\eps$, say $\eps = 1/1000$, and a family $\mathcal{C}$ of constraints which is $(1/1000,4)$-sufficient, for which Mathematica shows the output $\sum_{i \in [4]}d_i \log i \leq 0.999$.
The level of accuracy of the program is enough for us to know that certainly $g(A_1,\ldots,A_4) = \sum_{i \in [4]}d_i \log i < 1 - 1/2000$.
So in this case, we are done, and in fact since this number is less than our lower bound~(\ref{gbound1}) by some absolute constant, we see that $A_1,\ldots,A_4$ cannot be close to extremal.

If instead, given input $\eps=1/1000$, Mathematica shows an output $0.999 < \sum_{i \in [4]}d_i \log i \leq 1.001$, say, we need to be more careful.
In this case, we will write out the dual program of $\max \sum_{i \in [4]}d_i \log i$ subject to $\mathcal{C}$, which is a minimisation problem. We then exhibit a feasible solution to the dual which is at most $1+ \eps$.
By the weak duality theorem, we see that $\max \sum_{i \in [4]}d_i \log i \leq 1 + \eps$, as required.

\subsection{Linear constraints for general $r$}

To achieve~(1), we will first derive a set of linear constraints which apply for any number $r$ of colours.

\begin{lemma}\label{lem-generalineq}
	For all $\ep>0$ and integers $r\ge 4$, there exists an $n_0>0$ such that the following holds. Let $n,s,t \in \mathbb{N}$ be such that $n\ge n_0$ and $s+t\le r$. Also, let $A_1,\ldots,A_r$ be maximal sum-free subsets of $[n]$ such that $s$ of them are of type (b), $t$ of them are of type (c), and $(d_1,\ldots,d_{r})$ is their intersection vector. Then
	$$
	\sum_{i \in [r]}id_i \leq \frac{r}{2}-\frac{r-s-t}{10}-a+\ep \quad\text{and}\quad d_{r-1}+d_r \leq \frac{1}{2}+\ep.
	$$
\end{lemma}

\begin{proof}
	Let $n_0:=r/\ep$ and let $n \geq n_0$ be an integer. For the first inequality, by~\eqref{eq-defab}, we have
	\begin{align*}
		\sum_{i \in [r]}id_i&\leq \frac{s+t}{n}\cdot\ceil{ \frac{n}{2}}+\left(\frac{2}{5}+\frac{1}{n}\right)(r-s-t)-a \leq \frac{r}{2}-\frac{r-s-t}{10}-a+\ep,
	\end{align*}
	where the last inequality follows from $n \geq r/\ep$.  
	To prove the second part of the lemma, since $\ep\ge r/n\ge 1/n$, it suffices to show that the set $D_r\cup D_{r-1}$ is sum-free. Assume to the contrary that there exist $x_1,x_2,x_3\in D_r\cup D_{r-1}$ such that $x_1+x_2=x_3$. For every $i\in [3]$, define $I_i:=\{j: j\in [r]\text{ and }x_i\in A_j\}$ and $\overline{I_i}=[r]\setminus I_i$. Since $x_i\in D_r\cup D_{r-1}$, we have $|\overline{I_i}|\le 1$ for every $i\in [3]$, and therefore
$
	\left|\bigcap_{i\in[3]}I_i\right|\ge r-\sum_{i \in [3]}\left|\overline{I_i}\right|\ge r-3\ge 1$,
	where the last inequality follows from $r\ge 4$. Therefore, there exists an $i\in [r]$ such that $x_1,x_2,x_3\in A_i$, which contradicts $A_i$ being sum-free.
\end{proof}

We will use the next two simple facts repeatedly. We omit their proofs since they follow from the definitions of $a$ and type (c) sets.

\begin{observation}\label{lem-minelement}
Let $k \in [r]$ and suppose that $A_1,\ldots,A_k$ are of type (c). Then
\begin{itemize}
\item[(i)] $\min (A_i)\ge \ceil{n/2}-an$ for all $i \in [k]$.
\item[(ii)] $
\sum_{i\in [k]} d(A_i)\ge k\cdot \frac{\ceil{n/2}}{n}-a$.
\end{itemize}
\end{observation}

The next lemma concerns the size of the intersection of type (c) sets.

\begin{lemma}\label{lem-overlap}
	For all $\ep>0$ and positive integers $r$, there exists an $n_0>0$, such that the following holds. For every integer $n\ge n_0$, let $A_1,\ldots, A_r$ be maximal sum-free subsets of $[n]$ such that there is some $k \leq r$ for which $A_1,\ldots,A_k\subseteq [n]$ are type~(c) sets. Then
	$$
	d(\cap_{i\in [k]}A_i)\ge \frac{1}{2}-ka-\ep.
	$$
\end{lemma}
\begin{proof}
 Let $n_0:=k/\ep$ and let $n > n_0$ be an integer and $A_1,\ldots,A_r$ be subsets of $[n]$ as in the statement.
 	Let $A^*=\cap_{i\in [k]}A_i$. It suffices to show that $|A^*|\ge \ceil{n/2}-(k-1)-kan$. By Observation~\ref{lem-minelement}(i), for all $i\in [k]$, $\min(A_i)\ge \ceil{n/2}-an$, and therefore  $A_i\subseteq I_a := \left[\ceil{n/2}-an,n\right]$.
	Every integer $x$ in the subset $L:=I_a\setminus A^*$ lies in at most $k-1$ of the $A_i$'s.
		Assume to the contrary that $|A^*|< \ceil{n/2}-(k-1)-kan$. 
		Then
	\begin{align}\label{eq-overlap2}
	\nonumber \sum_{i\in[k]}|A_i|&\le k\cdot |A^*|+(k-1)|L|  = (k-1)|I_a|+|A^*|\\
	\nonumber &< (k-1)(n-(\lceil n/2\rceil-an)+1) + \lceil n/2\rceil - (k-1)-kan\\
	&= (k-1)n - (k-2)\lceil n/2\rceil - an.
	\end{align}
But Observation~\ref{lem-minelement}(ii) implies that
$
	\sum_{i\in[k]}|A_i|\ge k\ceil{n/2}-an$.
	Together with~\eqref{eq-overlap2}, we get $2(k-1)\ceil{n/2}<(k-1)n$, a contradiction. 
\end{proof}
Throughout the rest of the paper, given $a$ defined by~(\ref{eq-defab}), we will let
\begin{equation}\label{Idef}
J_1:=\left[\floor{\frac{n}{2}}-an\right],\quad J_2:=\left[\floor{\frac{n}{2}}-an+1,n\right],\quad\text{and}\quad J_3:=\left[\floor{\frac{n}{2}}-an+1, \floor{\frac{n}{2}}\right]
\end{equation}
and will refer to $J_1$, $J_2$, and $J_3$ as the \emph{first}, \emph{second}, and \emph{middle interval} respectively. Note that by definition $an$ is an integer, and the set $J_2\setminus J_3=I_2$ is a sum-free set of maximum size $\ceil{n/2}$. The following observation is a straightforward consequence of Observation~\ref{lem-minelement}(i) and the fact that the unique maximal sum-free subset of type~(b) is $\Odd$.

\begin{observation}\label{obs-allbut1inI2}
\begin{itemize}
\item[(i)]	If $A_i$ is a maximal sum-free set of type~(c), then $|A_i\setminus J_2| \leq 1$.
\item[(ii)]
	For all $\ep>0$ and positive integers $r$, there exists $n_0>0$ such that the following holds. Let $A_1,\ldots,A_{r}$ be maximal sum-free subsets of $[n]$ such that $s$ of them are of type (b), $t$ of them are of type (c), and $r-s-t$ of them are of type~(a). Then,
	$$
	d\left((D_0 \cup \ldots \cup D_{r-s-t}) \cap J_1\cap E\right), d\left( (D_s \cup \ldots \cup D_{r-t}) \cap J_1\cap O\right)\ge  1/4-a/2-\ep.
	$$
	\end{itemize}
\end{observation}

The final result in this subsection states some constraints involving the intervals $J_2$ and $J_3$.

\begin{lemma}\label{lem-middleelements}
	For all $\ep>0$ and positive integers $r$, there exists $n_0>0$ such that the following holds. Let $k\le r$ and $n\ge n_0$ be positive integers.
	Let $A_{1},\ldots,A_k\subseteq [n]$ be type~(c) maximal sum-free sets and define $q_i:=d(A_i\cap J_3)$ for all $i \in [k]$. Then  
$$\sum_{i\in [k]}q_i\le a+\ep\quad\text{and}\quad d((\cap_{i\in [k]}A_i)\cap I_2)\ge \frac{1}{2}-\sum_{i\in [k]}q_i-a-\ep \geq \frac{1}{2} - 2a - 2\eps.
$$
\end{lemma}
\begin{proof}
	Let $n_0:=r/\ep$ and let $n \geq n_0$ be an integer. Since $A_i$ has $q_in$ elements in $J_3$, we have that $\min(A_i)\le \floor{\frac{n}{2}}+1-q_in$ for all $i \in [k]$.
	Using Observation~\ref{lem-minelement}(ii) and the definition of type (c) sets, we have
	\begin{equation}\label{aieq}
	\frac{k}{n}\ceil{\frac{n}{2}} - a \leq \sum_{i \in [k]}d(A_i) = \sum_{i \in [k]}\frac{|A_i|}{n} \leq \sum_{i \in [k]}\frac{\min(A_i)}{n} \leq \frac{k}{n}\left(\ceil{\frac{n}{2}}+1\right) - \sum_{i \in [k]}q_i.
	\end{equation}
Thus	 $\sum_{i \in [k]}q_i \leq k/n+a \leq a+\eps$, proving the first inequality.
	To prove~(ii), let $B := I_2 \setminus \cap_{i\in [k]}A_i$. In other words, $B$ is the set of all elements in $I_2 = J_2\setminus J_3$ that are missing from at least one of the $A_1,\ldots,A_k$.
	By Observation~\ref{obs-allbut1inI2}(i),
	\begin{align}\label{eq-aiqi}
	\sum_{i\in [k]}d(A_i)&\le \sum_{i\in [k]} q_i+\sum_{i\in [k]} d(A_i\cap (J_2\setminus J_3))+\frac{k}{n} \le \sum_{i\in [k]}q_i+\frac{k}{n}\ceil{n/2}-d(B)+\ep.
	\end{align} 
	Thus
	$$
	d(B) \leq \sum_{i \in [k]}q_i + \frac{k}{n}\lceil n/2\rceil - \sum_{i \in [k]}d(A_i) + \eps \stackrel{(\ref{aieq})}{\leq} \sum_{i \in [k]}q_i+a+\eps,
	$$
	and the second required inequality follows.
\end{proof}

\subsection{The 4 colour case}

\emph{Proof of Lemma~\ref{45colours2} when $r=4$.}
Let $\eps>0$ and choose $\eps',n_0>0$ such that $1/n_0\ll\ep'\ll\ep \leq 1/100$.
	(we may assume the last inequality without loss of generality).
	Let $n \geq n_0$ be an integer and let $A_1,\ldots, A_4\subseteq [n]$ be maximal sum-free sets with intersection vector 
	$(d_1,\ldots,d_4)$ as defined in Problem~2. 
	We need to obtain a family of $(\eps,4)$-sufficient constraints.
	We have the following set of basic constraints which will be used throughout the proof:
	\begin{align}\label{eq-4c-basic}
	\begin{cases}
		d_i \ge 0 \text{ for all }i\in \{0,\ldots,4\},\nonumber\\
		\displaystyle\sum_{i\in\{0,\ldots,4\}}d_i\le 1,\nonumber\\
		d_3+d_4\le\ceil{\frac{n}{2}}/n\le\frac{1}{2}+\ep',\tag{C1}
		\end{cases}
	\end{align}
	where the last inequality follows from Lemma~\ref{lem-generalineq}.
	
	Suppose first that there is an $i \in [4]$ for which $A_i$ is of type (a).
	By Lemma~\ref{lem-generalineq}, 
		we have that
		\begin{equation}\tag{C0$^*$}\label{C0}
		\sum_{i \in [4]}id_i\le 19/10+\ep'.
		\end{equation}
		(The $^*$ denotes the fact that this inequality does not hold in Case~1 onwards.)
	The family $\lbrace \eqref{C0},\eqref{eq-4c-basic}\rbrace$ is $(\eps,4)$-sufficient (here and from now on, see Lemma~\ref{suff}).
	
Thus we can assume that all of $A_1,\ldots,A_4$ are of type~(b) or~(c).
Let $s$ be the number of $A_i$ of type~(b) and $t$ the number of $A_i$ of type~(c) (so $s+t=4$).
Define $a$ as in~(\ref{eq-defab}).
By Lemma~\ref{lem-generalineq} (with $\ep'$ playing the role of $\ep$), we have
	\begin{align}\label{eq-4c-idi}
	\sum_{i\in [4]}id_i\le 2-a+\ep'.\tag{C2}
	\end{align} 
Suppose that $a \geq 1/10$.
Then~\eqref{eq-4c-idi} implies that~\eqref{C0} holds. But, as we have seen, $\lbrace \eqref{C0},\eqref{eq-4c-basic}\rbrace$ is an $(\eps,4)$-sufficient family.
	Hence, throughout the rest of the proof, we can assume that
	\begin{align}\label{eq-4c-asmall}
		a<1/10.\tag{C3}
	\end{align}
	Given $a$, define $J_1,J_2,J_3$ as in~(\ref{Idef}).
By Observation~\ref{obs-allbut1inI2}(ii),
	\begin{align}\label{eq-4c-d0ds}
\left.
\begin{array}{r@{}l}
	d_0 &\geq d(D_0\cap J_1) \\
	d_s &\geq d(D_s\cap J_1)
\end{array}
\right\rbrace \ge \frac{1}{4} - \frac{a}{2} -\eps'.\tag{C4}
	\end{align}

	\begin{case}
		$s \in \lbrace 0,1\rbrace$.
	\end{case}

		Suppose first that $s=0$.
		Then $A_1,\ldots,A_4$ are all type~(c), so Lemma~\ref{lem-overlap} (with $\ep'$ playing the role of $\ep$) implies that 
		\begin{align}\label{eq-4c-alltypec-d4}
		d_4\ge \frac{1}{2}-4a-\ep'.\tag{C5}
		\end{align}
		The family $\lbrace\eqref{eq-4c-basic},\eqref{eq-4c-idi},\eqref{eq-4c-d0ds},\eqref{eq-4c-alltypec-d4}\rbrace$ is $(\eps,4)$-sufficient.
				
Suppose instead that $s=1$. Without loss of generality, let $A_4$ be the only type (b) set.
		By Lemma~\ref{lem-overlap} (with $\ep'$ playing the role of $\ep$), $d(\cap_{i \in [3]}A_i)\ge 1/2-3a-\ep'$. By Observation~\ref{obs-allbut1inI2}(i), $|\cap_{i\in[3]}A_i \setminus J_2| \leq 1$. Also, every element $x\in \cap_{i\in[3]}A_i$ is in $D_3$ if it is even, and $D_4$ if it is odd.
		Thus
		\begin{align}\label{eq-4c-d3}
		d_3&\ge d((\cap_{i\in[3]}A_i)\cap J_2)-d(O\cap J_2)\ge d(\cap_{i\in[3]}A_i)-\frac{1}{n}-\frac{1}{n}\cdot\ceil{(n-\floor{\frac{n}{2}}+an)/2}\tag{C6}\\
		\nonumber &\ge \frac{1}{2}-3a-\ep'-\frac{1}{n}-\frac{1}{4}-\frac{a}{2}-\frac{1}{2n}\ge \frac{1}{4}-\frac{7a}{2}-2\ep'.
		\end{align}
		The family $\lbrace\eqref{eq-4c-basic},\eqref{eq-4c-idi},\eqref{eq-4c-asmall},\eqref{eq-4c-d0ds},\eqref{eq-4c-d3}\rbrace$ is $(\eps,4)$-sufficient.
		This completes the proof of Case~1.

	\begin{case}
		$t \in \lbrace 0,1\rbrace$.
	\end{case}

	Suppose first that $t=0$.
		Then all of $A_1,\ldots,A_4$ are of type~(b), so $D_0$ is the set of evens, $D_4$ is the set of odds, and all the other $D_i$'s are empty. Therefore
		$$
		g(A_1,\ldots, A_4)=2\cdot d_4=\frac{2}{n}\cdot\ceil{\frac{n}{2}}\le 1+\ep',
		$$
		as required.
		Suppose instead that $t=1$.
		Then $D_0\cup D_1$ contains every even integer, and therefore
		\begin{align}\label{eq-4c-d0d1}
		d_0+d_1\ge \frac{1}{n}\cdot\floor{\frac{n}{2}}\ge \frac{1}{2}-\ep'.\tag{C7}
		\end{align}
		The family $\lbrace\eqref{eq-4c-basic},\eqref{eq-4c-idi},\eqref{eq-4c-asmall},\eqref{eq-4c-d0ds},\eqref{eq-4c-d0d1}\rbrace$ is $(\eps,4)$-sufficient.
		This completes the proof of Case~2.	Therefore, the only remaining case is the following.

	\begin{case}
		$s=t=2$. 
	\end{case}  

We will prove that the following constraints hold.
		\begin{align}
		d_2&\ge \frac{1}{2}-\frac{5a}{2}-2\ep',\tag{C8}\label{eq-4c-d2}\\
%		d_4&\ge 0.25-2a-\ep', \text{(I think we don't use this one)}\tag{C11}\label{eq-4c-d4}\\
		d_3&\le 3a+\ep'.\tag{C9}\label{eq-4c-22-d3}
		\end{align}
				Without loss of generality, we can assume that $A_1$ and $A_2$ are of type~(c). 
		We first prove~\eqref{eq-4c-d2}. By~\eqref{eq-4c-d0ds}, we have that $d(D_2\cap J_1)\ge 1/4-a/2-\ep'$. Therefore, we only need to show that $d(D_2\cap J_2)\ge 1/4-2a-\ep'$. 
		Let $B := A_1 \cap A_2 \cap (J_2\setminus J_3)$.
		Then $B \cap E \subseteq D_2 \cap J_2$.
		Lemma~\ref{lem-middleelements} applied with parameter $\eps'/4$ implies that $d(B) \geq 1/2-2a-\eps'/2$.
		Further, $J_2\setminus J_3$ is an interval of length $\lceil n/2\rceil$ so contains at most $n/4+1$ odd elements. Thus
		 $$
		d(D_2 \cap J_2) \geq d(B \cap E) = d(B) - d(B \cap O) \geq \frac{1}{2} - 2a - \frac{\ep'}{2} - d((J_2\setminus J_3) \cap O) \geq \frac{1}{4} - 2a - \eps'.  
		$$

\begin{comment}		
		For proving~\eqref{eq-4c-d4}, since $A_1\cap A_2\cap \{J_2\setminus J_3\}\cap O\subseteq D_4$, it suffices to show that  $d(\{A_1\cap A_2\cap \{J_2\setminus J_3\}\}\setminus \{E\cap \{J_2\setminus J_3\}\})\ge 0.25-2a-\ep'$. Similar to the proof of~\eqref{eq-4c-d2}, by Lemma~\ref{lem-middleelements} (with $\ep'/2$ playing the role of $\ep$), 
		\begin{align}\label{eq-4c-d4I2}
		d(A_1\cap A_2\cap\{J_2\setminus J_3\})\ge 0.5-2a-\ep'/2.
		\end{align}
		Also, $d(E\cap \{J_2\setminus J_3\})\le 0.25+\ep'/2 $, which, together with~\eqref{eq-4c-d4I2}, finishes the proof.
\end{comment}
	
We now prove~\eqref{eq-4c-22-d3}. 
By Observation~\ref{obs-allbut1inI2}(i), we have $|A_1\cap J_1|,|A_2 \cap J_1| \leq 1$. Thus 
$|D_3 \cap J_1| \leq 2$.
Further, we have $A_1 \cap A_2 \cap E \subseteq D_2$ and $A_1 \cap A_2 \cap O \subseteq D_4$.
In particular, $(A_1 \cap A_2) \cap D_3 = \emptyset$. Therefore
$$
d_3 = d(D_3 \cap J_1) + d(D_3\cap J_2) \leq \frac{2}{n} + d(J_2) - d(B) \stackrel{(\ref{Idef})}{\leq} \frac{2}{n} + 1-\frac{1}{n}\cdot\lfloor n/2\rfloor + a - \frac{1}{2} + 2a + \frac{\eps'}{2} \leq 3a + \eps',
$$
as required.
	But $\lbrace\eqref{eq-4c-basic},\eqref{eq-4c-idi},\eqref{eq-4c-asmall},\eqref{eq-4c-d2},\eqref{eq-4c-22-d3}\rbrace$ is an $(\eps,4)$-sufficient family.
This completes the proof of Case~3, the final case.
\hfill$\square$

\subsection{The $5$ colour case}

\emph{Proof of Lemma~\ref{45colours2} in the case $r=5$.}
	Let $\eps>0$ and choose $\eps',n_0>0$ such that $1/n_0\ll\ep'\ll\ep \leq 1/200$.
	(we may assume the last inequality without loss of generality).
	%	By applying Theorem~\ref{problem2}, with $\ep'$ playing the role of $\ep$, it suffices to show that $g(n,4)\le 1+\ep'$. In other words, our goal is to solve the following optimisation problem.
	%	\begin{align}
	%	\max_{\substack{(A_1,\ldots,A_r) \subseteq [n]^r:\\
	%			A_i \text{ sum-free for all }i\in[r]}}g(A_1,\ldots,A_r)\tag{*}
	%	\end{align}
	Let $n > n_0$ be an integer and let $A_1,\ldots, A_5\subseteq [n]$ be maximal sum-free sets with intersection vector 
	$(d_1,\ldots,d_5)$. 
	We have the following basic constraints which will be used throughout the proof.
	\begin{align}\label{eq-5c-basic}
	\begin{cases}
	&d_i \ge 0 \text{ for all }i\in \{0,\ldots,5\},\nonumber\\
	&\displaystyle\sum_{i\in\{0,\ldots,5\}}d_i\leq 1.\tag{D1}
	\end{cases}
	\end{align}
	Suppose first that $A_1,A_2$ are of type~(a).
	Now Lemma~\ref{lem-generalineq} (with $\ep'$ playing the role of $\ep$) implies that
		\begin{equation}\tag{D0$^*$}\label{D0}
	\sum_{i \in [5]}id_i \leq 23/10 + \eps'.
	\end{equation}
	(The $^*$ denotes the fact that this inequality does not hold in Case~1 onwards.)
	But $\lbrace\eqref{D0},\eqref{eq-5c-basic}\rbrace$ is an $(\eps,5)$-sufficient family.

	Let $s$ be the number of $A_i$ of type~(b) and $t$ the number of $A_i$ of type~(c). So we may assume that $s+t \in \lbrace 4,5\rbrace$.

		Define $a$ as in~(\ref{eq-defab}).
	Lemma~\ref{lem-generalineq} implies that  
	\begin{align}\label{eq-5c-idi}
	\sum_{i\in [5]}id_i\le
	\left\{
	\begin{array}{ll}
	5/2-a+\ep',&\mbox{\quad if }s+t=5, \\
	12/5-a+\ep',&\mbox{\quad if }s+t=4.
	\end{array}
	\right.\tag{D2}
	\end{align} 
	
	Suppose now that $s+t=5$ and $a\geq 1/5$; or $s+t=4$ and $a \geq 1/10$.
	Then~\eqref{eq-5c-idi} implies that~\eqref{D0} holds.
	But, as above, $\lbrace\eqref{D0},\eqref{eq-5c-basic}\rbrace$ is an $(\eps,5)$-sufficient family.
	Thus we may assume that
	\begin{equation}\label{eq-5c-absmall}
	\text{if } s+t=5,\text{ then }a<1/5;\quad \text{and if } s+t=4,\text{ then }a<1/10.\tag{D3}
	\end{equation}
	By Observation~\ref{obs-allbut1inI2}(ii), we have 
	\begin{eqnarray}\label{eq-5c-case1-d0ds}
	\left.
	\begin{array}{r@{}l}
    d_0 &\geq d(D_0\cap J_1 \cap E) \\
    d_s &\geq d(D_s\cap J_1 \cap \Odd)
  \end{array}
  \right\rbrace \ge \frac{1}{4} - \frac{a}{2} - \eps',\quad\text{if }s+t=5,\\
  \label{eq-5c-case2-d0ds}
  \left.
	\begin{array}{r@{}l}
    d_0+d_1 &\geq d((D_0 \cup D_1)\cap J_1 \cap E) \\
    d_s+d_{s+1} &\geq d((D_s \cup D_{s+1})\cap J_1 \cap \Odd)
  \end{array}
  \right\rbrace \ge \frac{1}{4} - \frac{a}{2} - \eps',\quad\text{if }s+t=4,
  \end{eqnarray}
		We work through some cases depending on the values of $(s,t)$, in increasing order of complexity.

\begin{case}
$5-s \leq 1$ or $5-t \leq 1$.
\end{case}
	In this case, we will see that
	\begin{equation}\label{d0+d1}\tag{D4}
	d_0+d_1 \geq 1/2-a-\eps'.
	\end{equation}
		Suppose first that $5-s \leq 1$. 
		Now at most one $A_i$ can contain an even number, so $d_0+d_1\ge |E|/n \ge 1/2-\ep' \geq 1/2 -a-\eps'$, as required.
		Suppose secondly that $5-t \leq 1$.
		Then summing the inequalities in each of~\eqref{eq-5c-case1-d0ds} and~\eqref{eq-5c-case2-d0ds} implies that $d_0+d_1 \geq 1/2-a-\eps'$.
		But $\lbrace\eqref{eq-5c-basic},\eqref{eq-5c-idi},\eqref{eq-5c-absmall},\eqref{d0+d1}\rbrace$ is an $(\eps,5)$-sufficient family.
		
		The remaining cases are $(s,t) \in \lbrace (1,3),(3,1),(2,3),(3,2),(2,2)\rbrace$.
			
	\begin{case}
		$(s,t) \in \lbrace (1,3), (3,1)\rbrace$.
	\end{case} 
		Suppose that $A_5$ is the only set of type~(a).
		By~\eqref{eq-5c-case2-d0ds}, we have that
		\begin{equation}
		d_0+d_1 \geq 1/4-a/2-\eps'.\tag{D5}\label{eq-5c-1o-d0d1}
		\end{equation}
		Suppose first that $(s,t)=(1,3)$. Then summing the inequalities in~\eqref{eq-5c-case2-d0ds} implies that
		\begin{equation}\label{d0d1d2}\tag{D6}
		d_0+d_1+d_2 \geq 1/2-a-2\eps'.
		\end{equation}
		But $\lbrace\eqref{eq-5c-basic},\eqref{eq-5c-idi},\eqref{eq-5c-absmall},\eqref{eq-5c-1o-d0d1},\eqref{d0d1d2}\rbrace$ is an $(\eps,5)$-sufficient family.
		
		Suppose secondly that $(s,t)=(3,1)$.
		Suppose that $A_4$ is the only set of type~(c).
		We will prove that the following inequalities hold.
		\begin{align}
		&d_1+d_2\ge 1/4-3a/2-3\ep',\tag{D7}\label{eq-5c-3o-d1d2}\\
		&d_3\le 1/4+3a/2+3\ep'.\tag{D8}\label{eq-5c-3o-d3}
		\end{align}
		
		Note that
		\begin{equation}\label{A4A5again}
		A_4 \cap E \subseteq D_1 \cup D_2\quad\text{and}\quad A_4 \cap \Odd \subseteq D_4 \cup D_5.
		\end{equation}
		Further, using Lemma~\ref{lem-overlap} and Observation~\ref{obs-allbut1inI2}(i), we have
		\begin{equation}\label{A4}
		d(A_4 \cap J_2) = d(A_4) - d(A_4\setminus J_2) \geq (1/2-a-\eps') - 1/n \geq 1/2-a-2\eps'.
		\end{equation}
		Therefore
		\begin{eqnarray*}
		d_1+d_2&\ge& d((D_1\cup D_2)\cap J_2)\stackrel{(\ref{A4A5again})}{\ge} d(A_4\cap E \cap J_2) \ge d(A_4\cap J_2)-d(J_2\cap \Odd)\\
		&\stackrel{(\ref{Idef}),(\ref{A4})}{\geq}& (1/2-a-2\eps') - (1/4+a/2+\eps') = 1/4-3a/2-3\eps',
		\end{eqnarray*}
		proving~\eqref{eq-5c-3o-d1d2}.	
		To prove~\eqref{eq-5c-3o-d3},
		\begin{align*}
		d_3 &= d(D_3 \cap J_1) + d(D_3 \cap J_2)\\
		&\leq d(J_1) - d((D_0 \cup D_1) \cap J_1) + d(J_2) - d((D_1 \cup D_2 \cup D_4 \cup D_5) \cap J_2)\\
		&\stackrel{(\ref{A4A5again})}{\leq} 1 - d((D_0 \cup D_1) \cap J_1) - d(A_4 \cap J_2) \stackrel{(\ref{eq-5c-case2-d0ds}),(\ref{A4})}{\leq} 1 - (1/4-a/2-\eps') - (1/2-a-2\eps')\\
		&= 1/4 + 3a/2 + 3\eps',
		\end{align*}
		as desired.
		We have that $\lbrace\eqref{eq-5c-basic},\eqref{eq-5c-idi},\eqref{eq-5c-absmall},\eqref{eq-5c-1o-d0d1},\eqref{eq-5c-3o-d1d2},\eqref{eq-5c-3o-d3}\rbrace$ is an $(\eps,5)$-sufficient family.

	\begin{case}
		$(s,t)=(2,3)$.
	\end{case} 

	We will assume that
		$A_1$ and $A_2$ are type~(b) sets (so $A_1=A_2=\Odd$), and $A_3,A_4,A_5$ are type~(c) sets.
		Let $q_i:=d(A_i\cap J_3)$, for every $i\in \lbrace 3,4,5\rbrace$. First, we will prove that the following constraints hold.
		\begin{align}	
		d_5&\ge 1/4-2a-2\ep',\label{eq-5c-2o3i-d5}\tag{D9}\\
		d_3&\le 1/4+3a+4\ep'.\label{eq-5c-2o3i-d3}\tag{D10}
		\end{align}
		We first prove~\eqref{eq-5c-2o3i-d5}. By Lemma~\ref{lem-middleelements}, we have
		\begin{align}
		\label{d5eq} d(D_5 \cap I_2) &\geq d(\Odd\cap (A_3 \cap A_4 \cap A_5) \cap I_2)\\
		\nonumber &\geq d((A_3 \cap A_4 \cap A_5) \cap I_2) - d(E \cap I_2)\\
		\nonumber &\geq (1/2-2a-\eps') - (1/4+\eps') = 1/4 - 2a - 2\eps',
		\end{align}
		so $d_5$ is certainly at least this quantity.
		We now prove~\eqref{eq-5c-2o3i-d3}.
		Notice that every element of $D_3$ lies in at least one type~(c) set,~i.e. $D_3 \subseteq A_3 \cup A_4 \cup A_5$. 
		So $|D_3 \cap J_1| \leq \sum_{3 \leq i \leq 5}|A_3 \cap J_1| = \sum_{3 \leq i \leq 5}|A_3\setminus J_2| \leq 3$ by Observation~\ref{obs-allbut1inI2}(i).
		Thus
		\begin{eqnarray*}
		d_3 &=& d(D_3 \cap J_2) + d(D_3 \cap J_1) \leq d(J_2) - d(D_5 \cap J_2) + 3/n\\
		&\stackrel{(\ref{Idef}),(\ref{d5eq})}{\leq}& (1/2+a+\eps')-(1/4-2a-2\eps') + 3/n \leq 1/4 + 3a + 4\eps',
		\end{eqnarray*}
		as required.

		Since $A_1=A_2=\Odd$, we have $E \subseteq D_0 \cup \ldots \cup D_3$.
Let $B := J_3 \cap E \cap (D_2 \cup D_3)$.		
		Assume now that $d(B)\le a/4$.
		We claim that, in this case, the following hold: 
						\begin{align}\label{eq-2o-d0d1-1}
		d_0+d_1 &\geq 1/4-a/4-2\eps'.\tag{D11}\\
		\label{eq-2o-d2-1}
		d_2&\ge 1/4-a/2-\ep'.\tag{D12}
		\end{align}
		Indeed, to see the first inequality, observe that		
				\begin{equation}\nonumber
		d_0+d_1\ge d(D_0 \cap J_1) + d(J_3 \cap E)-\frac{a}{4} \stackrel{(\ref{eq-5c-case1-d0ds})}{\ge} \left(\frac{1}{4}-\frac{a}{2}-\eps'\right)+\frac{\lfloor|J_3|/2\rfloor }{n}-\frac{a}{4} \geq \frac{1}{4}-\frac{a}{4}-2\eps'.
		\end{equation}
		The second is a consequence of~\eqref{eq-5c-case1-d0ds}.
		But $\lbrace\eqref{eq-5c-basic},\eqref{eq-5c-idi},\eqref{eq-5c-absmall},\eqref{eq-5c-2o3i-d5},\eqref{eq-5c-2o3i-d3},\eqref{eq-2o-d0d1-1},\eqref{eq-2o-d2-1}\rbrace$ is an $(\eps,5)$-sufficient family.

		The only case left is when $d(B)> a/4$. We claim that, in this case, the following hold:
				\begin{align}\label{eq-2o-d2-2}
		\quad d_2&\ge 1/4-a/4-2\ep'.\tag{D13}\\
		\label{eq-2o-d0-1}
		d_0 &\ge 1/4-a/2-\ep'.\tag{D14}
		\end{align}
		The second inequality is simply~(\ref{eq-5c-case1-d0ds}).
		To see why the first holds, let $m:=\min(B)$. Then 
		$$
		m< \floor{\frac{n}{2}}-2\cdot\left(\frac{an}{4}-1\right)\le \frac{n}{2}-\frac{an}{2}+2.
		$$
		Since $m\in E$, we know that $m\notin A_1\cup A_2$. In addition, since $m\in D_2\cup D_3$, without loss of generality, we can assume that $m\in A_3\cap A_4$. Therefore, by the definition of type~(c) sets, we see that $|A_3|,|A_4|\le m$. Thus $|A_3|+|A_4|< n-an+4$. So, recalling that $|A_1|=|A_2|=\lceil n/2\rceil$, the definition of $a$ implies that $|A_5| = 3\lceil n/2\rceil - an - |A_3| - |A_4| > n/2-4$.	Thus $\min(A_5) > n/2-4$. So $|A_5 \cap J_3| \leq 3$.
		But $E \cap D_3 \subseteq A_3 \cap A_4 \cap A_5$, so $|J_3 \cap E \cap D_3| \leq 3$.
		Since $d(B) > a/4$, we have $d(J_3 \cap E \cap D_2) > a/4-3/n$. Thus
		$$
		d_2 \geq d(J_1 \cap D_2) + d(J_3 \cap E \cap D_2) \stackrel{(\ref{eq-5c-case1-d0ds})}{\geq} 1/4-a/2-\eps' + a/4-3/n = 1/4-a/4-2\eps',
		$$
		as required.		
		Now $\lbrace\eqref{eq-5c-basic},\eqref{eq-5c-idi},\eqref{eq-5c-absmall},\eqref{eq-5c-2o3i-d5},\eqref{eq-5c-2o3i-d3}, \eqref{eq-2o-d2-2},\eqref{eq-2o-d0-1}\rbrace$ is an $(\eps,5)$-sufficient family.

	\begin{case}
		$(s,t)=(3,2)$.
	\end{case}
	We will assume that $A_1$, $A_2$ and $A_3$ are type~(b) sets (and so $A_4$ and $A_5$ are type~(c) sets).
	We will prove that the following constraints hold.
	\begin{align}
	&d_0\ge 1/4-a/2-\ep'.\tag{D15}\label{eq-5c-3o2i-d0}\\
	&d_2+d_5\ge 1/2-2a-\ep',\tag{D16}\label{eq-3o-d2d5}\\
	&d_1+d_3+d_4\le 1/4+5a/2+3\ep',\tag{D17}\label{eq-3o-d1d3d4}\\
	&d_2\le 1/4+a/2+\ep'.\tag{D18}\label{eq-3o-d2}
	\end{align}
	The first inequality~\eqref{eq-5c-3o2i-d0} is simply~\eqref{eq-5c-case1-d0ds}.
	For~\eqref{eq-3o-d2d5}, observe that $E \cap A_4 \cap A_5 \subseteq D_2$, and $\Odd \cap A_4 \cap A_5 = D_5$. So
	\begin{equation}\label{A4A5}
	A_4 \cap A_5 \subseteq D_2 \cup D_5\quad\text{and}\quad d_2+d_5=d(D_2\cup D_5)\ge d(A_4\cap A_5)\ge 1/2-2a-\ep',
	\end{equation}
	where we used Lemma~\ref{lem-overlap} for the final inequality.
	For~\eqref{eq-3o-d1d3d4}, we have
	\begin{eqnarray*}
	d_1+d_3+d_4 &\leq& d((D_1 \cup D_3 \cup D_4) \cap J_1) + d((D_1 \cup D_3 \cup D_4) \cap J_2)\\
	&\stackrel{(\ref{A4A5})}{\leq}& d(J_1) - d(D_0 \cap J_1) + d(J_2 \setminus (A_4 \cap A_5)) \\
	&=& d(J_1) - d(D_0 \cap J_1) + d(J_2) - d(A_4 \cap A_5) + d((A_4 \cap A_5) \setminus J_2)\\
	&\leq& 1 - (1/4-a/2-\eps') - (1/2-2a-\eps') + 1/n = 1/4 + 5a/2 + 3\eps'.
	\end{eqnarray*}
	The final inequality follows from~\eqref{eq-5c-case1-d0ds},~\eqref{A4A5} and Observation~\ref{obs-allbut1inI2}(i).
	
	Finally we will prove~\eqref{eq-3o-d2}.
	Observe that $E \setminus (A_4 \cup A_5) \subseteq D_0$ and $\Odd \setminus (A_4 \cup A_5) \subseteq D_3$. By Observation~\ref{obs-allbut1inI2}(i), we have that $|J_1 \cap (A_4 \cup A_5)|=|(A_4 \cup A_5)\setminus J_2| \leq 2$. Therefore $|(J_1 \cap (D_0 \cup D_3)| \geq |J_1|-2$ and so $|D_2 \cap J_1| \leq 2$. 
	Further, $\Odd \subseteq D_3 \cup D_4 \cup D_5$. In particular, $D_2 \cap J_2 \subseteq E \cap J_2$.
	Combining these facts, we see that
	\begin{align*}
	d_2 = d(D_2 \cap J_1) + d(D_2 \cap J_2) \leq 2/n + d(E \cap J_2) \leq 2/n + \frac{\lceil |J_2|/2\rceil}{n} \stackrel{(\ref{Idef})}{\leq} 1/4 + a/2 + \eps',
	\end{align*}
	as desired.
	But $\lbrace\eqref{eq-5c-basic},\eqref{eq-5c-idi},\eqref{eq-5c-absmall},\eqref{eq-5c-3o2i-d0},\eqref{eq-3o-d2d5},\eqref{eq-3o-d1d3d4},\eqref{eq-3o-d2}\rbrace$ is an $(\eps,5)$-sufficient family.

	\begin{case}
		$(s,t)=(2,2)$
	\end{case} 
	
	We will assume that $A_1, A_2$ are of type~(b); $A_3, A_4$ are of type~(c); and $A_5$ is of type~(a).
	Our immediate aim is to prove that the following inequalities hold.
		\begin{align}\label{eq-5c-22-d0d1}
		&d_0+d_1\ge 1/4-a/2-\ep'.\tag{D19}\\
		\label{eq-5c-22-d3-d2}		&d_3-d_2\le 1/4+4a+6\ep'.\tag{D20}
		\end{align}
		The first is a consequence of~\eqref{eq-5c-case2-d0ds}.
		We will now prove~\eqref{eq-5c-22-d3-d2}.
		This requires careful analysis of the small unstructured set $A_5$.
		Define sets
		\begin{equation}
		I'_1:=\left[\ceil{\frac{n}{4}}\right],\quad I''_1:=J_1\setminus I'_1=\left[\ceil{\frac{n}{4}}+1,\floor{\frac{n}{2}}-an\right],\nonumber
		\end{equation}
		\begin{align*}
		&X_1:=A_5\cap O\cap I'_1, &&X_2:=A_5\cap O\cap I''_1, &&&x_i := d(X_i) \ \text{for } i=1,2;\\ 
		&Y_o:=A_5\cap O\cap J_2, &&  Y_e:=A_5\cap E\cap J_2, &&&y_o := d(Y_o) \ \text{and } y_e := d(Y_e),
		\end{align*}
		\begin{equation*}
		S:=\{x+y:x,y\in X_2\cup (Y_o\cap J_3)\}.
		\end{equation*}
		Clearly, it suffices to show that
					\begin{align}
		&d_3\le 3/8+7a/2-x_2+5\ep' \quad\text{and}\label{eq-5c-22-d3}\\
		&d_2\ge 1/8-a/2-x_2 - \eps'.\label{eq-5c-22-d2}
		\end{align}		
		Let $Z := J_1 \cap (A_3 \cup A_4)$. By Observation~\ref{obs-allbut1inI2}(i), we have $|Z| \leq 2$.
		To prove~\eqref{eq-5c-22-d3}, we bound $d(D_3\cap J_1)$ and $d(D_3\cap J_2)$ separately.
		If $v \in D_3 \cap E$, then $v \in A_3 \cap A_4 \cap A_5$. So $D_3 \cap J_1 \cap E \subseteq Z$.
		Similarly, $D_3 \cap J_1 \cap \Odd \subseteq (A_5 \cap J_1) \cup Z$.
		Thus
		\begin{align}\label{eq-5c-22-d3i1}
		d(D_3\cap J_1)= d(D_3\cap J_1\cap E)+d(D_3\cap J_1\cap O)\le d(Z)+x_1+x_2\le x_1+x_2+\ep'.
		\end{align}
		Now, $A_3 \cap A_4 \cap \Odd \subseteq D_4 \cup D_5$. Further, $D_3 \cap E \subseteq A_5$.
		Therefore
		\begin{align}\label{eq-5c-22-d3i2}
		d(D_3\cap J_2)&= d(D_3\cap J_2\cap E)+d(D_3\cap J_2\cap O)\le y_e+d(J_2\cap O)-d(A_3\cap A_4\cap J_2\cap O)\nonumber\\
		&\le y_e+d(J_2\cap O)-(d(A_3\cap A_4\cap J_2)-d(J_2\cap E))\nonumber\\
		&\le y_e + d(J_2) - d(A_3 \cap A_4 \cap I_2) \le y_e+ (1/2+a+\eps')- (1/2-2a-2\eps') \nonumber\\
		&= y_e + 3a+3\eps',
		\end{align}
		where we used~(\ref{Idef}) and Lemma~\ref{lem-middleelements} for the final inequality.
		
				For every $s \in S$, we have that $s$ is even and at most $n$, and additionally $s \geq 2(\lceil n/4\rceil+1)\geq \lfloor n/2\rfloor-an+1$. Thus $S \subseteq J_2 \cap E$.
		Since $A_5$ is sum-free, we have that $Y_e \cap S = \emptyset$. So $|S|+|Y_e| \leq |J_2 \cap E|$.
		By the Cauchy-Davenport theorem (Theorem~\ref{CD}) applied to $S$, we have $|S| \geq 2|X_2 \cup (Y_o \cap J_3)| - 1$. Thus
		\begin{equation}
		\label{eq-5c-22-sumset} 2(|X_2|+|Y_o \cap J_3|) + |Y_e|  \leq |S|+|Y_e|+1 \leq |J_2 \cap E| + 1 \stackrel{(\ref{Idef})}{\leq} (1/4+a/2+\eps')n.
		\end{equation}
		
		Combining this with~\eqref{eq-5c-22-d3i1} and~\eqref{eq-5c-22-d3i2} we have
		\begin{align*}
		d_3&\le x_1+x_2+y_e+3a+4\ep'\stackrel{\eqref{eq-5c-22-sumset}}{\le} x_1+(1/4+a/2+\ep'-x_2)+3a+4\ep'\\
		&\le 3/8+7a/2-x_2+5\ep',
		\end{align*}
		where we used the trival bound $x_1n \leq \lceil|I_1'|/2\rceil$.
		This finishes the proof of~\eqref{eq-5c-22-d3}.
		
		For~\eqref{eq-5c-22-d2}, notice that
		$\Odd \setminus (A_3 \cup A_4 \cup A_5) \subseteq D_2$.
		 By definition, $J_1\setminus Z$ is disjoint from $A_3\cup A_4$. Therefore
		\begin{align*}
		d_2&\ge d(\Odd \cap I_1'') - d(\Odd \cap I_1'' \cap (A_3 \cup A_4)) - d(\Odd \cap I_1'' \cap A_5)\\
		&\geq d(\Odd \cap I_1'')-d(Z)-x_2\ge \frac{1}{n}\lfloor |I_1''|/2\rfloor -2/n-x_2\ge 1/8 -a/2-x_2-\ep',
		\end{align*}
		as required.
		We have proved~\eqref{eq-5c-22-d2} and hence~\eqref{eq-5c-22-d3-d2}.
		
		The remainder of the proof will be divided into two final subcases. First, suppose that $Y_o\subseteq J_3$. Now~\eqref{eq-5c-22-sumset} implies that $2(x_2+y_o)+y_e\le 1/4+a/2+\ep'$.
		Thus
		\begin{align}\label{eq-5c-22-a5}
		d(A_5) &=x_1+x_2+y_o+y_e\le d(\Odd \cap I'_1)+1/4+a/2+\ep' \le 3/8+a/2+2\ep'.
		\end{align}
		Lemma~\ref{lem-generalineq} implies that	$\sum_{i \in [4]} d(A_i)\le 2-a+\ep'$.
		Adding these, we see that
		\begin{align}\label{eq-5c-22-idi1}
		\sum_{i \in [5]} id_i=\sum_{i \in [5]} d(A_i)\leq 19/8-a/2+3\ep'.\tag{D21}
		\end{align}
		Now $\lbrace\eqref{eq-5c-basic},\eqref{eq-5c-absmall},\eqref{eq-5c-22-d0d1},\eqref{eq-5c-22-d3-d2},\eqref{eq-5c-22-idi1}\rbrace$ is an $(\eps,5)$-sufficient family.
		
	The final subcase is when $Y_o\setminus J_3\neq\emptyset$. Let $w\in Y_o\setminus J_3$. So $w \geq \lfloor n/2\rfloor+1$. We claim that the following inequality holds:
			\begin{align*}\label{eq-5c-22-d3-2}\tag{D22}
		d_3 &\leq 5/16 + 7a/2 + 5\eps'.
		\end{align*}
		To prove this, define
\begin{align*}
X'_1:=X_1\setminus \left\{ \ceil{\frac{n}{4}}\right\}\quad\text{and}\quad
	D:=
	\left\{
	\begin{array}{ll}
		\{w+x:x\in X'_1\}&\mbox{\quad if }w\in \left[\floor{\frac{n}{2}}+1,n-\ceil{\frac{n}{4}}\right], \\
		\quad\\
		\{w-x:x\in X'_1\}&\mbox{\quad if }w\in \left[n-\ceil{\frac{n}{4}}+1,n\right].
	\end{array}
	\right.
\end{align*}
	Now, $w \in \Odd$ and $X_1' \subseteq \Odd$, so every element of $D$ is even.
	In both cases it is easy to check that $D \subseteq E \cap J_2$.
	Since $A_5$ is sum-free, $D\cap Y_e=\emptyset$, and so $D$ and $Y_e$ are disjoint subsets of $E \cap J_2$. In particular,
		\begin{align}\label{x1ye}
		x_1 + y_e \leq |D|/n+y_e + 1/n\le d(J_2\cap E) + 1/n \stackrel{(\ref{Idef})}{\le} 1/4+a/2+\ep'.
		\end{align} 
This then implies that
		\begin{align*}
		d_3 &= d(D_3 \cap J_1) + d(D_3 \cap J_2) \stackrel{(\ref{eq-5c-22-d3i1}),(\ref{eq-5c-22-d3i2})}{\leq} x_1+x_2 + y_e + 3a + 4\eps'\\
		\nonumber &= \frac{1}{2}(x_1+y_e) + \frac{1}{2}(2x_2+y_e) + \frac{x_1}{2} + 3a + 4\eps' \stackrel{(\ref{eq-5c-22-sumset}),(\ref{x1ye})}{\leq} \frac{1}{4} + \frac{a}{2} + \frac{x_1}{2} + 3a + 5\eps'\\
		\nonumber &\leq \frac{5}{16} + \frac{7a}{2} + 5\eps',
		\end{align*}
		This proves the claim. But $\lbrace\eqref{eq-5c-basic},\eqref{eq-5c-idi},\eqref{eq-5c-absmall},\eqref{eq-5c-22-d0d1},\eqref{eq-5c-22-d3-d2},\eqref{eq-5c-22-d3-2}\rbrace$ is an $(\eps,5)$-sufficient family.
		This completes the proof of the final case, and hence completes the proof of Lemma~\ref{45colours2}.
\hfill$\square$

\section{Concluding remarks}\label{sec:conclude}
We determined $f(n,r)$ exactly when $r=2$ (Theorem~\ref{2colours}). It would be interesting to proceed from our stability result (Theorem~\ref{3colours}) and obtain an exact result for $r=3$, and characterise the extremal sets. It seems possible to extract a statement about stability from the proof of Theorem~\ref{45colours} by more careful analysis of the linear programs. That is, the following may be obtainable. For all $\eps >0$, as long as $n$ is a sufficiently large integer: if $r=4$ and $A \subseteq [n]$ is extremal, then one of $A \bigtriangleup \Odd$, $A\bigtriangleup I_2$ and $A \bigtriangleup (\Odd \cup I_2)$ has size at most $\eps n$; and if $r=5$ and $A \subseteq [n]$ is extremal, then $|A \bigtriangleup (\Odd \cup I_2)| \leq \eps n$.

It is also possible that the method used to prove Theorem~\ref{45colours} (namely finding sufficient linear constraints) can prove the analogous result for $r=6$.
The main obstacle is the fact that, among extremal $A_1,\ldots,A_6$, one cannot \emph{a priori} guarantee less than two type~(a) sets.
This leads to 18 different values of $(s,t)$ to consider.
Since the proof for $r=5$ was already very involved, we did not pursue this further.

Finally, for large $r$, the value of $f(n,r)$ and the structure of the extremal sets is completely open.

\section{Acknowledgements}

We are grateful to Tuan Tran for helpful discussions.

\medskip

{\footnotesize \obeylines \parindent=0pt

\begin{tabular}{lll}
Hong Liu, Maryam Sharifzadeh and Katherine Staden\\
Mathematics Institute\\
University of Warwick\\
Coventry\\
CV4 2AL\\
UK\\
\end{tabular}
}

\appendix

\section{}

It remains to prove that the families obtained in the proof of Lemma~\ref{45colours2} are indeed $(\eps,r)$-sufficient. Namely, we require that the following lemma holds.

\begin{lemma}\label{suff}
Given $\eps>0$, for $r \in \lbrace 4,5\rbrace$ there exists $\delta,n_0>0$ such that whenever $\delta \leq \eps' \leq 1/100$ is a real constant and $n \geq n_0$ is an integer and $A_1,\ldots,A_r$ are maximal sum-free subsets of $[n]$, we have that:
The following families (depending on $\eps'$) are $(\eps,4)$-sufficient.
\begin{align*}
&1) &&\lbrace\eqref{C0},\eqref{eq-4c-basic}\rbrace &&&\text{(Case 0)}\\
&2) &&\lbrace\eqref{eq-4c-basic},\eqref{eq-4c-idi},\eqref{eq-4c-d0ds},\eqref{eq-4c-alltypec-d4}\rbrace &&&\text{(Case 1)}\\
&3) &&\lbrace\eqref{eq-4c-basic},\eqref{eq-4c-idi},\eqref{eq-4c-asmall},\eqref{eq-4c-d0ds},\eqref{eq-4c-d3}\rbrace &&&\text{(Case 1)}\\
&4) &&\lbrace\eqref{eq-4c-basic},\eqref{eq-4c-idi},\eqref{eq-4c-asmall},\eqref{eq-4c-d0ds},\eqref{eq-4c-d0d1}\rbrace &&&\text{(Case 2)}\\
&5) &&\lbrace\eqref{eq-4c-basic},\eqref{eq-4c-idi},\eqref{eq-4c-asmall},\eqref{eq-4c-d2},\eqref{eq-4c-22-d3}\rbrace &&&\text{(Case 3)}
\end{align*}
The following families (depending on $\eps'$) are $(\eps,5)$-sufficient.
\begin{align*}
&6) &&\lbrace \eqref{D0}, \eqref{eq-5c-basic} \rbrace &&&\text{(Case 0)}\\
&7) &&\lbrace \eqref{eq-5c-basic}, \eqref{eq-5c-idi} , \eqref{eq-5c-absmall} ,\eqref{d0+d1} \rbrace &&&\text{(Case 1)}\\
&8) &&\lbrace \eqref{eq-5c-basic}, \eqref{eq-5c-idi} , \eqref{eq-5c-absmall} , \eqref{eq-5c-1o-d0d1} , \eqref{d0d1d2} \rbrace &&&\text{(Case 2)}\\
&9) &&\lbrace\eqref{eq-5c-basic},\eqref{eq-5c-idi},\eqref{eq-5c-absmall},\eqref{eq-5c-1o-d0d1},\eqref{eq-5c-3o-d1d2},\eqref{eq-5c-3o-d3}\rbrace &&&\text{(Case 2)}\\
&10) &&\lbrace\eqref{eq-5c-basic},\eqref{eq-5c-idi},\eqref{eq-5c-absmall},\eqref{eq-5c-2o3i-d5},\eqref{eq-5c-2o3i-d3},\eqref{eq-2o-d0d1-1},\eqref{eq-2o-d2-1}\rbrace &&&\text{(Case 3)}\\
&11) &&\lbrace\eqref{eq-5c-basic},\eqref{eq-5c-idi},\eqref{eq-5c-absmall},\eqref{eq-5c-2o3i-d5},\eqref{eq-5c-2o3i-d3}, \eqref{eq-2o-d2-2},\eqref{eq-2o-d0-1}\rbrace &&&\text{(Case 3)}\\
&12) &&\lbrace\eqref{eq-5c-basic},\eqref{eq-5c-idi},\eqref{eq-5c-absmall},\eqref{eq-5c-3o2i-d0},\eqref{eq-3o-d2d5},\eqref{eq-3o-d1d3d4},\eqref{eq-3o-d2}\rbrace &&&\text{(Case 4)}\\
&13) &&\lbrace\eqref{eq-5c-basic},\eqref{eq-5c-absmall},\eqref{eq-5c-22-d0d1},\eqref{eq-5c-22-d3-d2},\eqref{eq-5c-22-idi1}\rbrace &&&\text{(Case 4)}\\
&14) &&\lbrace\eqref{eq-5c-basic},\eqref{eq-5c-idi},\eqref{eq-5c-absmall},\eqref{eq-5c-22-d0d1},\eqref{eq-5c-22-d3-d2},\eqref{eq-5c-22-d3-2}\rbrace &&&\text{(Case 5)}.
\end{align*}
\end{lemma}

\begin{proof}
For $1), 3)$ and $4)$, taking $\eps'=1/100$ in Mathematica yields $\sum_{i \in [4]}d_i \log i < 1 - 1/1000$, so we are done in these cases.
Given a linear maximisation (\emph{primal}) program:
\begin{itemize}
\item[] Maximise $\bm{c}^\intercal \bm{d}$ subject to $A\bm{d} \leq \bm{b}$ and $\bm{d} \geq \bm{0}$,
\end{itemize}
the dual minimisation program is: 
\begin{itemize}
\item[] Minimise $\bm{b}^\intercal \bm{y}$ subject to $A^\intercal \bm{y} \geq \bm{c}$ and $\bm{y} \geq \bm{0}$.
\end{itemize}
\medskip
\noindent
\textbf{Family 2)}
Taking the program represented by 2) as the primal, we have

\begin{blockarray}{rrrrrrrr}
& a & $d_0$ & $d_1$ & $d_2$ & $d_3$ & $d_4$ \\
\begin{block}{c(rrrrrr)l}
 $y_1$ &  & 1 & 1 & 1 & 1 & 1 & $\leq 1$  \\
 $y_2$ & 1 &  &  &  &  &  & $\leq \frac{1}{10}$  \\
 $y_3$ & 1 &  & 1 & 2 & 3 & 4 & $\leq 2 + \eps'$  \\
 $y_4$ &  &  &  &  & 1 & 1 & $\leq \frac{1}{2} + \eps'$  \\
 $y_5$ & -1 & -1 &  &  &  &  & $\leq -\frac{1}{2}+\eps'$  \\
 $y_6$ & -4 &  &  &  &  & -1 & $\leq -\frac{1}{2} + \eps'$.  \\
\end{block}
\begin{block}{cccccccc}
 & $\geq 0$ & $\geq 0$ & $\geq 0$ & $\geq 1$ & $\geq \log 3$ & $\geq 2$ &   \\
\end{block}
\end{blockarray}
$\quad\quad\leftrightarrow\quad\quad$
\begin{blockarray}{rrr}
& $\bm{d}^\intercal$ & \\
\begin{block}{c(c)l}
 $\bm{y}$ & A & $\leq \bm{b}$  \\
\end{block}
\begin{block}{ccc}
 & $\geq \bm{c}^\intercal$ &   \\
\end{block}
\end{blockarray}

A feasible solution to the dual program is $\bm{y}^* = (0,0,\frac{\log 3}{3},0,0,\frac{4\log 3}{3}-2)$. The objective function value of the dual at $\bm{y}^*$ is
$$
\bm{b}^\intercal\bm{y}^* = (2+\eps')\frac{\log 3}{3} + \left(-\frac{1}{2}+\eps'\right)\left(\frac{4\log 3}{3}-2\right) = 1+ \left(\frac{5\log 3}{3}-2\right)\eps' \leq 1 + \eps' \leq 1 + \eps.
$$
By the weak duality theorem, any feasible solution $\bm{x}$ to the primal maximisation linear program satisfies $\bm{c}^\intercal \bm{x} \leq \bm{b}^\intercal \bm{y}^* \leq 1 + \eps$.
Thus the family in 2) is $(\eps,4)$-sufficient.

\medskip
\noindent
\textbf{Family 5)}
The family yields the following primal and dual linear programs.

\begin{blockarray}{rrrrrrrr}
& a & $d_0$ & $d_1$ & $d_2$ & $d_3$ & $d_4$ \\
\begin{block}{c(rrrrrr)l}
			$y_1$&  & 1 & 1 & 1 & 1 & 1 & $\le 1$ \\ 
			$y_{2}$&  &  &  &  & 1 & 1 &  $\le \frac{1}{2}+\ep'$\\
			$y_3$& 1 &  & 1 & 2 & 3 & 4 &  $\le 2+\ep'$ \\
			$y_4$& 1 &  &  &  &  &  &  $\le \frac{1}{10}$\\ 
			$y_{5}$& $-\frac{5}{2}$ &  &  & -1 &  &  &  $\le -\frac{1}{2}+2\ep'$\\
			$y_{6}$& -3 &  &  &  & 1 &  &  $\le \ep'$.\\
			\end{block}
\begin{block}{cccccccc}
			& $\ge 0$ & $\ge 0$ & $\ge 0$ & $\ge 1$ & $\ge\log 3$ & $\ge 2$ &\\
\end{block}
\end{blockarray}

A feasible solution to the dual program is $\bm{y}^* = (0,0,\frac{1}{2},0,0,\log 3 - \frac{3}{2})$.
The objective function value of the dual at $\bm{y}^*$ is
$$
\bm{b}^\intercal\bm{y}^* = (2+\eps')\cdot \frac{1}{2} + \eps'\left(\log 3 - \frac{3}{2}\right) = 1 + (\log 3 - 1)\eps' \leq 1 + \eps' \leq 1 + \eps.
$$
We are again done by the weak duality theorem.

\medskip
\noindent
Now we let $r=5$ and consider families 6)--14). For 6)--9), 13) and 14), taking $\eps'=1/100$ in Mathematica yields $\sum_{i \in [5]}d_i\log i < \frac{1}{4}\log 30 - 1/10^4$, so we are done in these cases.
It remains to consider 10)--12).

\medskip
\noindent
\textbf{Family 10)}
The family yields the following primal and dual linear programs.

\begin{blockarray}{rrrrrrrrr}
				& a & $d_0$ & $d_1$ & $d_2$ & $d_3$ & $d_4$ & $d_5$ & \\ 
\begin{block}{c(rrrrrrr)l}
					$y_1$&  & 1 & 1 & 1 & 1 & 1 & 1 & $\le 1$ \\ 
				$y_2$& 1 &  &  &  &  &  &  & $\le \frac{1}{5}$\\ 
				$y_3$& 1 &  & 1 & 2 & 3 & 4 & 5 & $\le \frac{5}{2}+\ep'$ \\
				$y_{4}$& -2 &  &  &  &  &  & -1 & $\le -\frac{1}{4}+2\ep'$\\
				$y_{5}$& -3 &  &  &  & 1 &  &  & $\le \frac{1}{4}+4\ep'$\\
				$y_{6}$& $-\frac{1}{4}$ & -1 & -1 &  &  &  &  & $\le -\frac{1}{4}+2\ep'$\\
				$y_{7}$& $-\frac{1}{2}$ &  &  & -1 &  &  &  & $\le -\frac{1}{4}+\ep'$.\\
			\end{block}
\begin{block}{ccccccccc}
& $\ge 0$ & $\ge 0$ & $\ge 0$ & $\ge 1$ & $\ge\log 3$ & $\ge 2$ & $\ge\log 5$ & \\
\end{block}
\end{blockarray}

A feasible solution to the dual program is
$$
\bm{y}^* = \left(4x,0,\frac{1}{2}-x,\frac{5}{2}-\log 5 - x, \log 3 - \frac{3}{2} - x, 4x,2x\right)
$$
where $x = \frac{3}{2}\log 3 - \log 5$.
The objective function value of the dual at $\bm{y}^*$ is
$$
\bm{b}^\intercal\bm{y}^* = \frac{1}{4}\log 30 + \left(-\frac{1}{2} + 22\log 3 - 14\log 5\right)\eps' < \frac{1}{4}\log 30 + 2\eps' \leq \frac{1}{4}\log 30 + \eps.
$$
So the family in 10) is $(\eps,5)$-sufficient.

\medskip
\noindent
\textbf{Family 11)}
The family yields the following primal and dual linear programs.

\begin{blockarray}{rrrrrrrrr}
				& a & $d_0$ & $d_1$ & $d_2$ & $d_3$ & $d_4$ & $d_5$ & \\ 
\begin{block}{c(rrrrrrr)l}
				$y_1$&  & 1 & 1 & 1 & 1 & 1 & 1 & $\le 1$ \\ 
				$y_2$& 1 &  &  &  &  &  &  & $\le \frac{1}{5}$\\ 
				$y_3$& 1 &  & 1 & 2 & 3 & 4 & 5 & $\le \frac{5}{2}+\ep'$ \\
				$y_{4}$& -2 &  &  &  &  &  & -1 & $\le -\frac{1}{4}+2\ep'$\\
				$y_{5}$& -3 &  &  &  & 1 &  &  & $\le \frac{1}{4}+4\ep'$\\
				$y_{6}$& $-\frac{1}{4}$ &  &  & -1 &  &  &  & $\le -\frac{1}{4}+2\ep'$\\
				$y_{7}$& $-\frac{1}{2}$ & -1 &  &  &  &  &  & $\le -\frac{1}{4}+\ep'$\\
			\end{block}
\begin{block}{ccccccccc}
				& $\ge 0$ & $\ge 0$ & $\ge 0$ & $\ge 1$ & $\ge\log 3$ & $\ge 2$ & $\ge\log 5$ & \\
\end{block}
\end{blockarray}

A feasible solution to the dual is
$$
\bm{y}^* = \left( 4x, 0, \frac{1}{2}-x, \frac{5}{2}-\log 5 - x, \log 3-\frac{3}{2}-x,4x,4x\right)
$$
where $x = 2\log 3 - \frac{4}{3}\log 5$.
The objective function value of the dual at $\bm{y}^*$ is
$$
\bm{b}^\intercal\bm{y}^* = \frac{1}{4}\log 30 + \left(-\frac{1}{2} + 6\log 3 - \frac{10}{3}\log 5\right)\eps' < \frac{1}{4}\log 30 + 2\eps' \leq \frac{1}{4}\log 30 + \eps.
$$
So the family in 11) is $(\eps,5)$-sufficient.

\medskip
\noindent
\textbf{Family 12)}
The family yields the following primal and dual linear programs.

\begin{blockarray}{rrrrrrrrr}
				& a & $d_0$ & $d_1$ & $d_2$ & $d_3$ & $d_4$ & $d_5$ & \\ 
\begin{block}{c(rrrrrrr)l}
			$y_1$&  & 1 & 1 & 1 & 1 & 1 & 1 & $\le 1$ \\ 
			$y_2$& 1 &  & 1 & 2 & 3 & 4 & 5 & $\le \frac{5}{2}+\ep'$ \\
			$y_3$& 1 &  &  &  &  &  &  & $\le \frac{1}{5}$\\ 	
			$y_4$& $-\frac{1}{2}$ & -1 &  &  &  &  &  & $\le -\frac{1}{4}+\ep'$\\ 	
			$y_{5}$& -2 &  &  & -1 &  &  & -1 & $\le -\frac{1}{2}+\ep'$\\
			$y_{6}$& $-\frac{5}{2}$ &  & 1 &  & 1 & 1 &  & $\le \frac{1}{4}+3\ep'$\\
			$y_7$& $-\frac{1}{2}$ &  &  & 1 &  &  &  & $\le \frac{1}{4}+\ep'$\\
			\end{block}
\begin{block}{ccccccccc}
			& $\ge 0$ & $\ge 0$ & $\ge 0$ & $\ge 1$ & $\ge\log 3$ & $\ge 2$ & $\ge\log 5$ & \\
\end{block}
\end{blockarray}

A feasible solution to the dual is
$$
\bm{y}^* = \left( y_1,\frac{\log 5}{5}-\frac{y_1}{5},0,y_1,0,\log 3-\frac{3}{5}\log 5 - \frac{2}{5}y_1, 1-\frac{2}{5}\log 5 -\frac{3}{5}y_1\right)
$$
where $y_1=1/5$.
The objective function value of the dual at $\bm{y}^*$ is
$$
\bm{b}^\intercal\bm{y}^* = \frac{1}{4}\log 30 + \left(3\log 3-2\log 5+1-y_1\right)\eps' < \frac{1}{4}\log 30 + \eps' \leq \frac{1}{4}\log 30 + \eps.
$$
So the family in 12) is $(\eps,5)$-sufficient.
This completes the proof.
\end{proof}

\end{document}